\documentclass[10pt]{article}
\usepackage{color}
\usepackage{amssymb}
\usepackage{amsthm,array,amssymb,amscd,amsfonts,latexsym, url}
\usepackage{amsmath}

\newtheorem{theo}{Theorem}[section]
\newtheorem{prop}[theo]{Proposition}
\newtheorem{claim}[theo]{Claim}
\newtheorem{lemm}[theo]{Lemma}

\newtheorem{rema}[theo]{Remark}
\newtheorem{Defi}[theo]{Definition}
\newtheorem{ex}[theo]{Example}

\newtheorem{conj}[theo]{Conjecture}

\title{Universally defined cycles I}
\author{Claire Voisin}
\date{}

\begin{document}
\maketitle

\begin{center} {\it To the memory of Jacob Murre}
\end{center}
\begin{abstract} We introduce and study the notion of universally defined cycles of smooth varieties of dimension $d$, and prove that they are given by polynomials in the Chern classes. A similar result is proved for universally defined cycles on products of smooth varieties. We also state a conjectural explicit form for universally defined cycles on powers of smooth varieties, and provide some steps towards establishing it.
 \end{abstract}
\section{Introduction}
The  Franchetta conjecture, which has now at least two  proofs (see \cite{mestrano}, \cite{harer})
asserts that for  any line bundle $\mathcal{C}$ on the universal curve
$\mathcal{C}\rightarrow U$,  where $U$ is a Zariski open set of the moduli space
of curves of genus $g\geq2$, the restriction  of $\mathcal{L}$ to the fibers
$\mathcal{C}_b$, for any point $b$ of $U$, is a multiple of the canonical bundle.

If we want to prove a similar result for higher dimensional manifolds, we are faced to several difficulties.
The main problem is the following: there is no moduli space of manifolds of given dimension. Even if we restrict
to canonically polarized manifolds, where such moduli spaces exist, some components
consist of just one point (rigid manifolds) so the above statement
certainly fails.

This paper studies the notion of  {\it universally defined cycles} introduced in \cite{voisindiagonal} and establishes
in this setting   an analogue of the Franchetta conjecture for higher dimensional manifolds. In the case of curves, the
analogue of the notion we introduce
would be roughly  the data  of a line bundle $\mathcal{L}_g$  on the ``universal curve over
$\mathcal{M}_g$'' for each $g$, satisfying
some mild compatibility properties relating the curves of various genera, using the degenerations
of curves of genus $g$ to curves of  genus $g'$ with $g'< g$. This notion has some similarity with the functorial Chow groups considered by Mumford \cite{mumford}.  Then the conclusion would be that
the restriction of these line bundles to fibers (that is curves of genus
 $g$ defined over a field) is  a rational multiple (independent of $g$) of the canonical bundle.

The precise definition  is as follows. Here Chow groups are taken with $\mathbb{Q}$-coefficients.
\begin{Defi}\label{defiuniversintro}
A universally defined cycle on
smooth complex varieties of dimension $n$ is the data
of a cycle
$\mathcal{Z}(\pi)\in {\rm CH}(\mathcal{X})$ for each
 smooth morphism $\pi: \mathcal{X}\rightarrow B$ of relative dimension $n$, where $\mathcal{X}$ and $B$ are smooth quasi-projective.
These data should satisfy the following axioms:

(i) If $r:B'\rightarrow B$ is a morphism, with $B'$ smooth,
and $r':\mathcal{X}':=\mathcal{X}\times_B B'\rightarrow  \mathcal{X}$, $\pi':\mathcal{X}'\rightarrow  B'$ are the two natural morphisms,
$$ \mathcal{Z}(\pi')={r'}^*\mathcal{Z}(\pi)
\,\,{\rm in}\,\, {\rm CH}(\mathcal{X}').$$

(ii)  If $\mathcal{X}'\subset \mathcal{X}$ is a Zariski open subset, and
$\pi':\mathcal{X}'\rightarrow B$ is the composition of the inclusion and of the morphism
$\pi:\mathcal{X}\rightarrow B$,
$$\mathcal{Z}(\pi')={\mathcal{Z}(\pi)}_{\mid \mathcal{X}'}
\,\,{\rm in}\,\, {\rm CH}(\mathcal{X}').$$
\end{Defi}

\begin{rema}\label{remaintro24} {\rm 1)  We could a priori work over any  field $K$
and smooth algebraic varieties $\mathcal{X},\,B$ over $K$. It might be however that our results need more assumptions on $K$.

2)  We will also  use the notation  $\mathcal{Z}(\mathcal{X})$ although the cycle depends on the morphism
$\pi$. When $B$ is a point, that is, $X$ is  a variety defined over a field, we will systematically  use the notation
$\mathcal{Z}({X})\in{\rm CH}(X)$.

3) In the definition above, smoothness of the morphism $\pi$ is essential, as well as  the absence of any projectivity assumption on
$\pi$.}
\end{rema}
\begin{ex}{\rm The Chern classes of the relative tangent bundle $T_\pi:=T_{\mathcal{X}/B}$
are universally defined, and more generally  any fixed polynomial in the relative Chern classes
is universally defined.
}
\end{ex}

Our first main result  in this paper is the following:
\begin{theo}\label{theointro1}
Let $\mathcal{Z}$ be a universally defined cycle on varieties of dimension $d$. Then
there exists a polynomial $P$ with rational coefficients in $d$ variables
$c_1,\,\ldots,\,c_d$, such that for any smooth  quasi-projective variety $X$ over $\mathbb{C}$,
$\mathcal{Z}(X)=P(c_1(X),\ldots,\,c_d(X))\,\,{\rm in}\,\,{\rm CH}(X)$.
\end{theo}
\begin{rema}{\rm Our definition involves families $\mathcal{X}\rightarrow B$ and the conclusion concerns only the constant families $X\rightarrow {\rm pt}$. One may wonder if one could  weaken the notion of universally defined cycle by working only with varieties
defined over a field (including function fields), the axiom (i) being replaced by
Fulton specialization. We have not been able to prove Theorem \ref{theointro1} with these weaker data.}
\end{rema}
\begin{rema}{\rm As we will see in Theorem \ref{theosansdiag}, the polynomial $P$ is in fact uniquely determined by $\mathcal{Z}$, assuming it has weighted degree $\leq d$ in the variables $c_i$ (where the degree of $c_i$ is $i$).
}
\end{rema}
Let us give a sketch of the proof for curves (the proof is much easier in this case):
The idea is the following:  Any smooth curve $C$ can be imbedded in $\mathbb{P}^3$,
say by a morphism
$i:C\rightarrow \mathbb{P}^3$,  and,  choosing a  large
enough positive integer $d$, there is a curve $C_d$ which is
the complete intersection  of two surfaces of degree $d$,  containing $i(C)$ as a component.
The universal (quasi-projective) smooth complete intersection curve of two surfaces of degree $d$
in $\mathbb{P}^3$ has  very simple Chow groups. So for the smooth part of $C_d$ as above,
there exists a coefficient $\alpha_d$ which is a rational number
such that
$\mathcal{Z}({C_{d,reg}})=\alpha_d c_1(K_{C_{d,reg}})$ in ${\rm CH}^1(C_{d,reg})$. Restricting
this equality to
$i(C)\setminus (C\cap {\rm Sing}\,C_d)$, the
localization exact sequence shows
that $\mathcal{Z}(C)=\alpha_d K_C+z'$ in ${\rm CH}^1(C)$, where $z'$ is supported on the intersection $C\cap {\rm Sing}\,C_d$.
Testing this equality on curves
for which there is no relations  in ${\rm CH}^1(C_{d,reg})$
between $K_C$ and points in $C\cap {\rm Sing}\,C_d$,
this  implies that $\alpha_d$ does not depend on $d$. Finally, an extra
trick (eg a monodromy argument on the points of $C\cap {\rm Sing}\,C_d$) shows that the cycle $z'$ has to be $0$ for any curve $C$.

Let us emphasize the following point (see also Remark \ref{remaintro24}, 3)): the reason why we get in the case of curves
a stronger result than the Franchetta conjecture, (since our conclusion is that
we get the same multiple of the canonical bundle for all genera)  is due to the fact that
we work with quasi-projective varieties, without properness assumption
on the morphism.
As it appears in the above sketch of proof, this allows to hide in axiom (ii) above many degenerations (where we remove the singular
points of the singular fibers), and this is why we can compare what happens for various genera.

 Theorem  \ref{theointro1} has a natural generalization
   concerning universally defined cycles  on products of varieties $X_1\times\ldots\times X_l$ of respective fixed dimensions $d_1,\ldots,\,d_l$.
 We mean by this the assignment of a cycle
 $\mathcal{Z}(\mathcal{X}_1,\ldots,\mathcal{X}_l)\in{\rm CH}(\mathcal{X}_1\times \ldots\times \mathcal{X}_l)$
 for each data of $l$ smooth morphisms $\mathcal{X}_1\rightarrow B_1,\ldots, \mathcal{X}_l\rightarrow B_l$ to smooth quasi-projective basis $B_i$, satisfying compatibility conditions as in Definition \ref{defiuniversintro} (see Definition \ref{defiaxiom} for more detail).
 We will prove
 \begin{theo}\label{theosansdiag} Let $\mathcal{Z}$ be a universally defined cycle on products of $r$
smooth varieties of dimension $d_1,\ldots,\,d_r$. Then there exists a uniquely defined polynomial $P_{\mathcal{Z}}$ with rational coefficients in the variables $c_{1,1},\ldots, c_{d_1,1},\ldots,c_{1,r},\ldots, c_{d_r,r}$, of weighted degree $\leq d_i$ in each set of variables $c_{1,i},\ldots, c_{d_i,i}$ (the weight of
$c_{j,i}$ being $j$), such that
for any  given  smooth varieties $X_1,\ldots,X_r$ over $\mathbb{C}$, of respective dimensions
$d_1,\ldots,\,d_r$, the following equality
\begin{eqnarray}\label{eqpourcyclepol}
\mathcal{Z}(X_1,\ldots,X_r)=P_{\mathcal{Z}}(p_1^*c_1(X_1),\ldots,\,p_1^*c_{d_1}(X_1),\,\ldots,\,
p_r^*c_1(X_r),\ldots,\,p_r^*c_{d_r}(X_r))
\end{eqnarray}
holds in ${\rm CH}(X_1\times\ldots\times X_r)$,
where  $p_i$ denotes the projection from $X_1\times\ldots\times X_r$ to $X_i$.
\end{theo}

 Theorem \ref{theosansdiag} is motivated by the following conjecture \ref{theavecpowerintro}, for which
we first need  to introduce
universally defined cycles on $k$-th powers of $n$-dimensional algebraic varieties.
We refer to Section \ref{sectionavecpuissance} for the precise definition, but let us just say that such a universally defined
cycle
$\mathcal{Z}$  is the assignment
of a cycle $\mathcal{Z}(\mathcal{X})\in{\rm CH}(\mathcal{X}^{k/B})$ for any smooth morphism $\mathcal{X}\rightarrow B$, where
$B$ is also smooth, satisfying the same axioms as in Definition
\ref{defiuniversintro}.
A typical example is given by
the diagonal $\Delta_{I,\mathcal{X}}$ determined by any partition $I$ of
$\{1,\ldots,\,k\}$ (see Section \ref{sectionavecpuissance} for the precise definition).
The present paper presents the first steps toward proving  the following statement, which, when $k=1$,  is Theorem
\ref{theointro1}.
\begin{conj}\label{theavecpowerintro}  For any universally defined cycle on
$k$-th powers of varieties of dimension $d$, there exists a unique collection of polynomials
$P_I$ indexed by partitions $I$ of $\{1,\ldots,\,k\}$, with the following properties: denoting by  $l(I)$ the length of the partition $I=\{I_1,\ldots,\,I_{l(I)}\}$,
$P_I$ is a polynomial  in the variables
 $c_{1,1},\ldots,\,c_{d,1},\ldots,\,c_{1,l(I)},\ldots,\,c_{d,l(I)}$, which is of weighted degree $\leq d$ in each set
 of variables $c_{1,s},\ldots,\,c_{d,s}$, and,  for each smooth  variety
 $X$ of dimension $d$ over $\mathbb{C}$, the equality
 \begin{eqnarray} \label{eqformulegeneralepourstand} \mathcal{Z}(X)=\sum_I\Delta_{I,X*}P_I({\rm pr}_1^*c_{1}(X),\ldots,\,{\rm pr}_1^*c_{d}(X),\ldots,\,{\rm pr}_{l(I)}^*c_{1}(X),\ldots,\,{\rm pr}_{l(I)}^*c_{d}(X) )
 \end{eqnarray}
holds in ${\rm CH}(X^k)$.
\end{conj}
In formula (\ref{eqformulegeneralepourstand}), we see $\Delta_{I,X}$ as a proper morphism from
$X^{l(I)}$ to $X^k$.

In the third  section of the paper, we will  establish the uniqueness statement in  Conjecture  \ref{theavecpowerintro} (Cf. Theorem  \ref{propunique}).

  Conjecture \ref{theavecpowerintro} had been unwisely announced as a forthcoming theorem  in the paper \cite{voisindiagonal} where we first introduced the formalism of universally defined cycles,  but its proof is still incomplete. It  would   have  many consequences
 that are described in \cite{voisinforthcoming} and also in \cite{voisindiagonal}. The first  consequence would be  an easy alternative proof of the Ellingsrud-G\"{o}ttsche-Lehn theorem \cite{EGL}. The proofs of \cite[Section 5]{voisindiagonal} concerning the Chow rings of Hilbert schemes of $K3$ surfaces depend on this statement and are therefore incomplete. In particular, as observed by Maulik,  Conjecture \ref{theavecpowerintro} would have major   applications to the  Beauville-Voisin conjecture on the Chow ring of Hilbert schemes of $K3$ surfaces (see \cite{beau}, \cite{voisinpamq}), obtained by  proving  the Lehn commutation relations (see \cite{lehn}) in the Chow setting. In the paper \cite{Mauliknegut} by Maulik and Negut, these relations  have  been given a direct proof.

Let us say that a universally defined cycle on powers of varieties is standard if it takes the form (\ref{eqformulegeneralepourstand}). Conjecture \ref{theavecpowerintro} says that any universally defined cycle on powers of varieties is standard. In  Section \ref{sectiontrick}, we establish   the following result, which provides a strong motivation for  Conjecture \ref{theavecpowerintro}:
\begin{theo}\label{corocohchow}  Let $\mathcal{Z}$ be  a universally defined cycle on powers
of $d$-dimensional smooth varieties, which is standard. If $\mathcal{Z}(X)$ is cohomologous to $0$ for any
smooth $d$-dimensional variety $X$ defined over $\mathbb{C}$, $\mathcal{Z}(X)=0$ in ${\rm CH}(X)$ for any
such $X$.
\end{theo}

\vspace{0.5cm}

{\bf Thanks.}
{\it  I thank Lie Fu, Marc Levine, Davesh Maulik and Charles Vial for interesting discussions related to this paper. This research started at IAS during the year 2014-2015, was continued at ETH-ITS during the year 2017, and finally concluded during my stay at MSRI for the program ``Birational geometry and moduli spaces", Spring  2019. I am indebted to these various institutions for providing a stimulating research environment.}
\section{Universally defined cycles on products\label{sectionsanspuissance}}
\subsection{Proof of Theorem \ref{theointro1}}
 This section is devoted to the proof of Theorem \ref{theointro1}. It is a particular case of Theorem \ref{theosansdiag} which will be proved in the next section  and whose proof needs extra ingredients.  The proof of Theorem \ref{theointro1} will use several preparatory results.
The following Proposition
 \ref{lefirstred} will allow to reduce Theorem \ref{theointro1} to the case of complete intersections, which is studied in Proposition \ref{proCIr=1}.
 Let $X$ be a variety and $E\rightarrow X$ be a rank $k$ vector bundle  on $X$ generated by a
 finite dimensional space $W$ of sections. Let $G=\mathbf{G}(k,W)$ be the Grassmannian of $k$-dimensional vector subspaces of $W$. Let $\mathcal{C}\subset G\times \mathbb{P}(E^*)$ be the set
 of pairs
 $$([V],e),\,V\subset W,\,\,e\in \mathbb{P}(E_x^*),$$ such that
 $$e(\sigma_x)=0\,\,\forall \sigma\in V,$$
 and let
 $\mathcal{C}_X$ be its image in $G\times X$ via the natural projection $\pi:\mathbb{P}(E^*)\rightarrow X$. Let us denote by $\mathcal{Q}$ the rank $k$ quotient bundle on $G$, so that $\mathcal{Q}^*$ is the rank $k$ subbundle on $G$ with fiber $V$ over the point $[V]$, and by
$p_X$, $p_G$  the projections from $G\times X$ to $X$ and $G$ respectively. The hypersurface $\mathcal{C}_X$ is thus the universal determinantal
 hypersurface associated to the universal evaluation morphism
 $$p_G^*\mathcal{Q}^*\rightarrow p_X^*E$$
 and $n:=\pi_{\mid \mathcal{C}}:\mathcal{C}\rightarrow \mathcal{C}_X$ is a natural  desingularization of $\mathcal{C}_X$.

 \begin{prop}\label{lefirstred} If ${\rm dim}\,W$ is large compared to ${\rm dim}\,X$,
 the pull-back and restriction composite map
 ${\rm CH}(X)\stackrel{p_X^*}{\rightarrow }{\rm CH}(G\times X)\rightarrow {\rm CH}((G\times X)\setminus \mathcal{C}_X)$ is injective.
\end{prop}
\begin{proof} Let  $h\in{\rm CH}(\mathbb{P}(E^*))$ be the first Chern class of the line bundle $\mathcal{O}_{\mathbb{P}(E^*)}(1)$.
We observe that
$\mathcal{C}$ is the zero-set of the tautological section of $\mathcal{Q}\boxtimes \mathcal{O}_{\mathbb{P}(E^*)}(1)$ on $G\times \mathbb{P}(E^*)$. Furthermore, as
$W$ is generated by sections, $\mathcal{C}$ is smooth of the right codimension $k$: more precisely,
via the projection to $\mathbb{P}(E^*)$, $\mathcal{C}$ is fibered over $\mathbb{P}(E^*)$ in Grassmannians $G(k,W^{\perp e})$ for hyperplanes $W^{\perp e}\subset W$, hence it is smooth and furthermore the restriction map
${\rm CH}(G\times\mathbb{P}(E^*))\rightarrow {\rm CH}(\mathcal{C})$ is surjective.
We thus conclude that the class $C$ of $\mathcal{C}$ in ${\rm CH}(G\times \mathbb{P}(E^*))$ is
$c_k(\mathcal{Q}\boxtimes \mathcal{O}_{\mathbb{P}(E^*)}(1))$ and that, denoting $j$ the inclusion of
$\mathcal{C}$ in $ G\times \mathbb{P}(E^*)$, ${\rm Im}\,(j_*:{\rm CH}^*(\mathcal{C})\rightarrow
{\rm CH}^{*+k}(G\times \mathbb{P}(E^*)))$ is equal to
${\rm Im}\,(C:{\rm CH}^*(G\times \mathbb{P}(E^*)) \rightarrow
{\rm CH}^{*+k}(G\times \mathbb{P}(E^*)))$.
Finally, as we are working with rational coefficients, the natural map
$n_*:{\rm CH}_*(\mathcal{C})\rightarrow {\rm CH}_*(\mathcal{C}_X)$ is surjective, and thus we conclude, using the localization exact sequence, that
${\rm Ker}\,( {\rm CH}(G\times X)\rightarrow  {\rm CH}((G\times X)\setminus \mathcal{C}_X))$ is equal to
\begin{eqnarray}\label{eqmap} {\rm Im}\,(\pi_*\circ C:{\rm CH}^*(G\times \mathbb{P}(E^*)) \rightarrow
{\rm CH}(G\times X)^{*+1}).
\end{eqnarray}
The space ${\rm CH}^*(G\times \mathbb{P}(E^*))$ is generated over the $\mathbb{Q}$-algebra ${\rm CH}^*(G\times X)$ by the classes $h^j$, for $j=0,\ldots, k-1$. As $\pi_*$ is a
${\rm CH}^*(G\times X)$-linear morphism by the projection formula, we conclude that
$ {\rm Im}\,(\pi_*\circ C)$ is generated over ${\rm CH}^*(G\times X)$ by the classes
$\pi_*(C\cdot h^j)$.  Proposition \ref{lefirstred} thus  follows from the following
\begin{claim}\label{claim} If $\alpha\in{\rm CH}(X)\subset  {\rm CH}(X\times G)$ has the property that $\alpha$
belongs to the ideal generated by the elements  $\pi_*(C\cdot h^j)$, $j=0,\ldots,\,k-1$, then
$\alpha=0$.
\end{claim}
Let us now prove the claim.  We use the same notation  $c_i(\mathcal{Q})$ and $h$ for the pull-back of these classes to
$G\times \mathbb{P}(E^*)$. Then $C=\sum_{i=0}^{i=k} c_i(\mathcal{Q})h^{k-i}$
and it follows that

\begin{eqnarray}\label{eqclasses} \pi_*(C\cdot h^j)=\sum_{i=0}^{i=k} c_i(\mathcal{Q})s_{j-i+1}(E^*)
\end{eqnarray}
for $j=0,\ldots,k-1$, where $s_i(E^*)$ denotes the $i$-th Segre class of $E^*$.
In particular,
\begin{eqnarray}\label{eqclasses5}\pi_*(C\cdot h^j)=c_{j+1}(\mathcal{Q})+P_j,\end{eqnarray} where $P_j$ is a polynomial
in the variables  $c_i(\mathcal{Q})$ with $i\leq j$, with coefficients in ${\rm CH}(X)$, and $P_0\in{\rm CH}(X)$.

 As ${\rm dim}\,W\gg {\rm dim}\,X$ and we can obviously restrict ourselves to  considering cycles of degree
$\leq {\rm dim}\,X$, there are no nontrivial relations in these degrees between the classes
$c_j(\mathcal{Q})$ for $j>0$. In other words, we can do as if ${\rm CH}(G)$ were the
 polynomial ring with generators
$Y_1=c_1(\mathcal{Q}),\ldots,\,Y_k=c_k(\mathcal{Q})$ and thus ${\rm CH}(G\times X)={\rm CH}(X)[Y_1,\ldots, Y_k]$. We know by (\ref{eqclasses5}) that
$\pi_*(C\cdot h^j)=Y_{j+1}+P_j$, where $P_j$ is a polynomial
in the $Y_{i}$ with $i\leq j$.
We apply  then the following elementary lemma to $R={\rm CH}(X)$, $a=\alpha$:
\begin{lemm} Let $R$ be a unitary commutative ring and
let $G_j=Y_j+P_j(Y_1,\ldots,Y_{j-1})\in R[Y_1,\ldots,\,Y_k]$, for $j=1,\ldots, k$. Then
for any relation $\sum_jF_j G_j=a$, with
$F_j\in R[Y_1,\ldots,\,Y_k]$ and  $a\in R$, one has $a=0$.
\end{lemm}
\begin{proof} Indeed, from the structure of the polynomials
$G_j$, we  construct inductively  $r_1,\ldots,r_k\in R$ such that
$G_j(r_1,\ldots, r_k)=0$ for all $j$, by the formulas $r_j=-P_j(r_1,\ldots,r_{j-1})$, $r_0=P_0$. Then $\sum_jF_j(r_\cdot)G_j(r_\cdot)=0=a$.
\end{proof}
Claim \ref{claim} is now proved and the proof of Proposition \ref{lefirstred} is now finished.
\end{proof}
The proof of Theorem \ref{theointro1}  will now be based on the following Proposition
\ref{proCIr=1}, which is
a stronger version of Theorem \ref{theointro1} for complete intersections.  For a given integer $d$, consider the Grassmannian
$\mathbf{G}(d,N)$, where $N\geq 2d+1$. From now on, the integer $N$ will be fixed, large enough compared to $d$. More precisely, we choose $N$ in such a way that any
$d$-dimensional smooth variety can be embedded in $\mathbf{G}(d,N)$.
For any given integer $l$, we consider the universal family of
$d$-dimensional complete intersections of type $(l,\ldots,l)$ in $\mathbf{G}(d,N)$. It is constructed as follows: Consider the Grassmannian $\mathbf{G}(k_d, V_{d,l})$ parameterizing
$k_d$-dimensional vector subspaces of
$V_{d,l}:=H^0(\mathbf{G}(d,N),\mathcal{L}^{\otimes l})$, where $\mathcal{L}$ is the Pl\"ucker line bundle
and $k_d:={\rm dim}\,\mathbf{G}(d,N)-d=d(N-d-1)$. The number $k_d$ is the codimension of a $d$-dimensional subvariety of $\mathbf{G}(d,N)$. A $d$-dimensional complete intersection in $\mathbf{G}(d,N)$ of type $(l,\ldots,\,l)$ is thus defined
by a vector subspace $V\subset V_{d,l}$ of dimension $k_d$.
Consider the following variety
$\mathcal{X}_{d,N,l}\subset \mathbf{G}(k_d, V_{d,l})\times \mathbf{G}(d,N)$
\begin{eqnarray}\label{eqdefXNL}
\mathcal{X}_{d,N,l}=\{([V],x)\in \mathbf{G}(k, V_{d,l})\times \mathbf{G}(d,N),\,v(x)=0\,\forall v\in V\}
.
\end{eqnarray}
Note that $\mathcal{X}_{d,N,l}$ is smooth and irreducible of codimension $k_d$ in $\mathbf{G}(k, V_{d,l})\times \mathbf{G}(d,N)$, being fibered over $\mathbf{G}(d,N)$, via the second projection ${\rm pr}_2$,  into Grassmannians $\mathbf{G}(k_d,V'_{d,l})$ for some hyperplanes
$V'_{d,l}\subset V_{d,l}$. Let
$\mathcal{X}^0_{d,N,l}\subset \mathcal{X}_{d,N,l}$ be the dense Zariski open set where
$\pi_{d,N,l}:={\rm pr}_1:\mathcal{X}_{d,N,l}\rightarrow \mathbf{G}(k_d, V_{d,l})$ is smooth. The fibers of the restriction \begin{eqnarray}\label{eqdefidNl}\pi_{d,N,l}^0:\mathcal{X}^0_{d,N,l}\rightarrow \mathbf{G}(k_d, V_{d,l})\end{eqnarray} are thus $d$-dimensional smooth {\it quasi-projective} complete intersections.
We will use the notation $c_i(\pi_{d,N,l}^0):=c_i(\mathcal{X}^0_{d,N,l}/\mathbf{G}(k_{d}, V_{d,l}))\in{\rm CH}(\mathcal{X}^0_{d,N,l})$.
 We denote below $B:=\mathbf{G}(k_{d},V_{d,l})$.

\begin{prop}\label{proCIr=1} Let $\mathcal{Z}$ be a universally defined cycle of degree $\leq d$ on  varieties of dimension $d$. Then there exists
a uniquely defined polynomial $P\in \mathbb{Q}[c_{1},\ldots,c_{d}]$, of weighted degree
$\leq d$ in the variables $c_{j}$ (of weighted degree $j$),  such that for  any $l$ large enough,
\begin{eqnarray}\label{eqCINl}
\mathcal{Z}(\pi_{d,N,l}^0)= P(c_1(\pi_{d,N,l}^0),\ldots,\,c_{d}
(\pi_{d,N,l}^0))
\end{eqnarray} in ${\rm CH}(\mathcal{X}^0_{d,N,l})$,
\end{prop}
\begin{rema} \label{remadu31juil} {\rm  The result is very easy  after restriction to the general fiber of $\pi_{d,N,l}^0$ or equivalently $\pi_{d,N,l}$. However we actually want the equality to hold in ${\rm CH}^*(\mathcal{X}^0_{d,N,l})$,  which  is much bigger than the Chow ring of the general fiber of $\pi_{d,N,l}^0$ because of the classes coming from the base. Indeed, we are working with cycles of codimension  $*\leq d$, and the complement of the Zariski open set $\mathcal{X}^0_{d,N,l}\subset \mathcal{X}_{d,N,l}$ is $d+1$, so in these degrees, ${\rm CH}^*(\mathcal{X}_{d,N,l})\cong {\rm CH}^*(\mathcal{X}^0_{d,N,l})$.}
\end{rema}
We  first show how Proposition \ref{proCIr=1} implies Theorem \ref{theointro1}.
\begin{proof}[Proof of Theorem \ref{theointro1}] Let $\mathcal{Z}$ be a universally defined cycle
on   varieties of dimension $d$. As noticed in Remark \ref{remasurledeg}, we can assume
 $\mathcal{Z}$ of  codimension  $\leq d$, as otherwise the statement is trivial. Let then
$P$ be the polynomial given by Proposition \ref{proCIr=1}. It suffices to show that for
 any smooth variety $X$ of  dimension $d$, $\mathcal{Z}(X)=P(c_i(X))$ in ${\rm CH}(X)$.
 We embed $X$ in $\mathbf{G}(d,N)$. For large $l$, $X$ is schematically defined by degree $l$ equations in
 $\mathbf{G}(d,N)$. Let
 $V_{d,l,X}\subset V_{d,l}$ be the space of sections of $\mathcal{L}^{\otimes l}$ vanishing along $X$.
 Then by differentiation along $X$, we get a natural surjective morphism
 \begin{eqnarray}
 \label{eqdiffmor} V_{d,l,X}\otimes\mathcal{O}_X\rightarrow N_{X/\mathbf{G}(d,N)}^*(l),
 \end{eqnarray}
 which makes $V_{d,l,X}$ into a generating space of sections of the vector bundle $ N_{X/\mathbf{G}(d,N)}^*(l)$.
 Let $\mathcal{C}_{X}\subset \mathbf{G}(k_{d}, V_{d,l,X})\times X$ be the corresponding degeneracy locus as discussed before Proposition \ref{lefirstred}.
 The points $([V],x)$ of $\mathcal{C}_{X}$ are the couples where the differentiation
 morphism
 $V_{d,l,X}\rightarrow N_{X/\mathbf{G}(d,N)}^*(l)_{|x}$ of  (\ref{eqdiffmor})  is not an isomorphism  at the point $x$, or equivalently, where the variety defined by the polynomials in $V$ is not locally isomorphic
 to $X$.
 The complement
 $(\mathbf{G}(k_{d}, V_{d,l,X})\times X)\setminus \mathcal{C}_{X}$ is thus open in a family of smooth complete intersections
 in $\mathbf{G}(d,N)$. More precisely,
 consider the inclusion
 $i:\mathbf{G}(k_{d}, V_{d,l,X})\subset \mathbf{G}(k_{d},V_{d,l})$. Then we have an open immersion
 $$(\mathbf{G}(k_{d}, V_{d,l,X})\times X)\setminus \mathcal{C}_{X}\subset i^*\mathcal{X}^0_{d,N,l}$$
 over $\mathbf{G}(k_{d}, V_{d,l,X})$,
 where $i^*\mathcal{X}^0_{d,N,l}$ is the universal family $\mathcal{X}^0_{d,N,l}$ of (\ref{eqdefidNl}), base-changed
 via $i$ to
 $\mathbf{G}(k_{d}, V_{d,l,X})$.

 Let $p_{X}$ denote  the projection from $(\mathbf{G}(k_{d}, V_{d,l,X})\times X)\setminus \mathcal{C}_{X}$ to $X$. Applying axioms (i) and (ii) of Definition \ref{defiaxiom}, we conclude that
 $$ p_{X}^*\mathcal{Z}(X)=
 \mathcal{Z}(\pi_{d,N,l}^0)_{\mid (\mathbf{G}(k_{d}, V_{d,l,X})\times X)\setminus \mathcal{C}_{X}}$$
 in ${\rm CH}((\mathbf{G}(k_{d}, V_{d,l,X})\times X)\setminus \mathcal{C}_{X})$.
  We now use Proposition \ref{proCIr=1} and observe that
  \begin{eqnarray}\label{eqditqueciuni}  c_i( \pi_{d,N,l}^0)_{\mid (\mathbf{G}(k_{d}, V_{d,l,X})\times X)\setminus \mathcal{C}_{X}})=c_i(T_{(\mathbf{G}(k_{d}, V_{d,l,X})\times X)\setminus \mathcal{C}_{X}/\mathbf{G}(k_{d}, V_{d,l,X}})=p_X^*c_i(X).\end{eqnarray}
   We then conclude that
 \begin{eqnarray}
 p_{X}^*\mathcal{Z}(X)=p_{X}^*P(c_i(X))\,\,{\rm in}\,\,{\rm CH}((\mathbf{G}(k_{d}, V_{d,l,X})\times X)\setminus \mathcal{C}_{X}).
 \end{eqnarray}
   By  Proposition \ref{lefirstred}, the morphism
 $p_{X}^*$ is injective and we thus conclude that
 $$\mathcal{Z}(X)=P(c_i(X))\,\,{\rm in}\,\,{\rm CH}( X).$$
 \end{proof}
 \begin{rema}{\rm What  has been used in a crucial way in this proof is the fact
 that the relative Chern classes are universally defined, hence, for  a morphism given by a
 projection $X\times B\rightarrow B$, they are pulled-back from $X$, and this remains true after restriction to an open set of $X\times B$ (this  is formula (\ref{eqditqueciuni})).
 }
 \end{rema}

 We  now start the proof of  Proposition \ref{proCIr=1}.
 We will need the following intermediate result:
 \begin{prop} \label{lealapacdenori} There exists a  smooth projective (not irreducible) variety  $X_0$, all of whose components  are  of dimension $d$, and a morphism  $f:X_0\rightarrow  \mathbf{G}(d,N)$,
  which satisfy the property that there are no nonzero cohomological
 polynomial relations $P(c_i(X_0), c'_j)=0$ in $H^{2*}(X_0,\mathbb{Q})$ in degree $*\leq d$, where the $c'_j$'s are the pull-backs to $X_0$ of the classes $c_j(\mathcal{Q})\in H^{2j}(\mathbf{G}(d,N),\mathbb{Q})$.
 \end{prop}
 \begin{rema}{\rm It is clear that it suffices to prove the statement for polynomials of  maximal degree $*=d$.  }\end{rema}
For the proof of this proposition,  we first establish the following statement.
\begin{lemm}\label{le1du13aout} There exist abelian varieties $A_i$ of dimension $d$, and morphisms
$$f_i: A_i\rightarrow \mathbf{G}(d,N)\times \mathbf{G}(d,N)$$
such that, denoting
$$c_{j,1}:={\rm pr}_1^*c_j(\mathcal{Q}),\,c_{j,2}:={\rm pr}_2^*c_j(\mathcal{Q})\in H^{2j}(\mathbf{G}(d,N)\times \mathbf{G}(d,N),\mathbb{Q}),$$
there is no nonzero polynomial $P$ of weighted degree $d$ in the variables $(x_{j,1},\,x_{j',2})$, such that the relation
\begin{eqnarray}\label{eqrelpolabelian} P(f_i^*c_{j,1},\,f_i^*c_{j',2})=0\,\,{\rm in}\,\,H^{2d}(A_i,\mathbb{Q}),
\end{eqnarray}
is satisfied for all $i$.
\end{lemm}

\begin{proof} According to \cite{milnor}, \cite{thom}, there are no nonzero linear relations between the Chern numbers
of products  $$\mathbb{P}^I:=\mathbb{P}^{r_1}\times\ldots\times\mathbb{P}^{r_{l(I)}}$$
of projective spaces indexed by partitions $I=\{r_1\leq\ldots\leq r_s\}$ of $d$, where $\sum _sr_s=d$.
Each $\mathbb{P}^I$ can be imbedded into $\mathbf{G}(d,N)$ by taking an adequate set of global sections of its tangent bundle $T_{\mathbb{P}^I}$. We denote by $g_I:\mathbb{P}^I\rightarrow \mathbf{G}(d,N)$ such an embedding. Then
$g_I^*\mathcal{Q}=T_{\mathbb{P}^I}$ so there is  no nonzero degree $d$ polynomial relation between the classes
$g_I^*c_j(\mathcal{Q})$ which holds in $H^{2d}(\mathbb{P}^I,\mathbb{Q})$  for all $I$.
Fixing two partitions $I,\,J$ of $d$, we know that the degree $2d$ cohomology $H^{2d}(\mathbb{P}^I\times \mathbb{P}^J,\mathbb{Q})$ is generated by classes of products of projective spaces $i_{K,I,J}:\mathbb{P}^K\subset \mathbb{P}^I\times \mathbb{P}^J$.
It follows now from Poincaré duality that there are no nonzero degree $d$ polynomial relations between the Chern classes $i_{K,I,J}^*((g_I,g_J)^*c_{j,1}(\mathcal{Q})),\,(g_I,g_J)^*c_{j',2}(\mathcal{Q}))$, holding for all $i_{K,I,J}$. Hence we almost proved   Lemma \ref{le1du13aout}, except that our varieties $\mathbb{P}^K$ mapped to $\mathbf{G}(d,N)\times \mathbf{G}(d,N)$ via  $(g_I,g_J)\circ i_{K,I,J}$ are not abelian varieties, but products of projective spaces. As any product $\mathbb{P}^K$ of projective spaces  is dominated by an abelian variety $r_K:A_K\rightarrow \mathbb{P}^K$ of the same dimension, we can replace the morphisms $(g_I,g_J)\circ \circ i_{K,I,J}: \mathbb{P}^K\rightarrow \mathbf{G}(d,N)\times \mathbf{G}(d,N)$ by $(g_I,g_J)\circ \circ i_{K,I,J}\circ r_K:A_K\rightarrow \mathbf{G}(d,N)\times \mathbf{G}(d,N)$, and the proof is finished.
\end{proof}
\begin{proof}[Proof of Proposition \ref{lealapacdenori}] Consider the abelian varieties $A_i$ and the morphisms
$$f_i=({f_{i,1},f_{i,2}}): A_i{\rightarrow} \mathbf{G}(d,N)\times \mathbf{G}(d,N)$$
of Lemma \ref{le1du13aout}. First of all,  we can assume by an easy twist argument, that in Lemma \ref{le1du13aout},  each $\mathcal{Q}_{1,i}:=f_{i,1}^*\mathcal{Q}$ is a very  ample vector  bundle on $A_i$. We now define
$B_i\subset \mathbb{P}(\mathcal{Q}_{1,i})$ to be a general complete intersection of $d-1$ hypersurfaces in
$|\mathcal{O}_{\mathbb{P}(\mathcal{Q}_{1,i})}(m)|$, where $m$ is taken large enough, independent of $i$. Using the fact that the tangent bundle of $A_i$ is trivial, the Chern classes of the tangent bundle of $\mathbb{P}(\mathcal{Q}_{1,i})$ equal those of  the relative tangent bundle $T_{\mathbb{P}(\mathcal{Q}_{1,i})/A_i}$. The relative tangent bundle of $$\pi_i: \mathbb{P}(\mathcal{Q}_{1,i})\rightarrow A_i$$
  is described by the Euler exact sequence

$$0\rightarrow \mathcal{O}_{\mathbb{P}(\mathcal{Q}_{1,i})}\rightarrow \pi_i^*\mathcal{Q}_{1,i}^*\otimes \mathcal{O}_{\mathbb{P}(\mathcal{Q}_{1,i})}(1)\rightarrow T_{\mathbb{P}(\mathcal{Q}_{1,i})/A_i}\rightarrow 0,$$
so that
$$c(T_{\mathbb{P}(\mathcal{Q}_{1,i})/A_i})=c(\pi_i^*\mathcal{Q}_{1,i}^*\otimes \mathcal{O}_{\mathbb{P}(\mathcal{Q}_{1,i})}(1))\,\,{\rm in}\,\,H^{2*}(\mathcal{Q}_{1,i},\mathbb{Q}).$$
Combined with the normal bundle sequence of $B_i$ in $\mathbb{P}(\mathcal{Q}_{1,i})$, we
get the following formula for the total Chern class of $B_i$ (where $h_i:=c_1(\mathcal{O}_{\mathbb{P}(\mathcal{Q}_{1,i})}(1))_{\mid B_i}$):

\begin{eqnarray}\label{eqlaborieusede13aout} c(B_i)= (1-mh_i+\ldots + (-1)^dm^{d}h_i^d)^{d-1}c(\pi_i^*\mathcal{Q}_{1,i}^*\otimes \mathcal{O}_{\mathbb{P}(\mathcal{Q}_{1,i})}(1)).\end{eqnarray}
Denoting by $\pi'_i:B_i\rightarrow A_i$ the restriction of $\pi_i$ to $B_i$, we have
\begin{eqnarray} \label{eqextrpourfin13aout} \pi'_{i*} h_i^l=(-1)^lm^{d-1} s_l(\mathcal{Q}_{1,i}),
\end{eqnarray} and we conclude from
(\ref{eqlaborieusede13aout}), (\ref{eqextrpourfin13aout}) and the projection formula that for  any polynomial $P(x_l,y_{l'})$  of weighted degree $d$ in the variables $x_l$ of degree $l$ and $y_{l'}$ of degree $l'$, there is a  polynomial $P'(x_l,y_{l'})$  of the same weighted degree such that, for all $i$,
$$\pi'_{i*}(P(c_l(B_i), c_{l'}({\pi'_i}^*\mathcal{Q}_{2,i}))=P'(c_{l}(\mathcal{Q}_{1,i}),c_{l'}(\mathcal{Q}_{2,i}))\,\,{\rm in}\,\,H^{2*}(A_i,\mathbb{Q}).$$
Furthermore, if $m$ is taken large enough, the linear map $P\rightarrow P'$ is  a bijection of the space of weighted homogeneous polynomials of the given degree. The disjoint union $X_0$ of the $B_i$, equipped with the vector bundles ${\pi'_i}^*\mathcal{Q}_{2,i}$, thus satisfy the conclusion of Proposition \ref{lealapacdenori}.
\end{proof}

 \begin{proof}[Proof of Proposition \ref{proCIr=1}] Let $\mathcal{Z}$ be  a universally defined cycle
  of codimension $\leq d$  on varieties of dimension $d$.
Recall that $\mathcal{X}_{d,N,l}^0$ is the universal smooth $d$-dimensional complete intersection of type $(l,\ldots,l)$ in $\mathbf{G}(d,N) $. It is Zariski open in
 the variety $\mathcal{X}_{d,N,l}\subset \mathbf{G}(k_{d}, V_{d,l})\times \mathbf{G}(d,N) $ introduced in (\ref{eqdefXNL}). Via the second projection
 ${\rm pr}_{\mathbf{G}(d,N)}:\mathcal{X}_{d,N,l}\rightarrow \mathbf{G}(d,N)$, $\mathcal{X}_{d,N,l}$ is
 a fibration into Grassmannians $\mathbf{G}(k_{d}, V_{d,l}^{\perp e})$, where
 $V_{d,l}^{\perp e}\subset V_{d,l}$ is the hyperplane of sections vanishing at $e\in \mathbf{G}(d,N) $, hence we conclude, with the notation  $$B:=\mathbf{G}(k_d,V_{d,l}),$$ that
 ${\rm CH}(B)\otimes {\rm CH}(\mathbf{G}(d,N))={\rm CH}(B\times \mathbf{G}(d,N))$
 maps surjectively to ${\rm CH}( \mathcal{X}_{d,N,l})$ (and in fact also injectively in the small codimensions we are considering).
 As the restriction map ${\rm CH}(\mathcal{X}_{d,N,l})
 \rightarrow {\rm CH}(\mathcal{X}_{d,N,l}^0)$ is surjective, the
 same is true with $\mathcal{X}_{d,N,l}$ replaced by $\mathcal{X}_{d,N,l}^0$. In fact, as the critical locus of $\pi_{d,N,l}$ has codimension $d+1$ in $\mathcal{X}_{d,N,l}$,  the last
  restriction map is injective as well on cycles of codimension  $\leq d$, and similarly for the restriction
  maps ${\rm CH}(\mathbf{G}(k_{d}, V_{d,l}))\rightarrow {\rm CH}(\mathbf{G}(k_{d}, V_{d,l}^{\perp e}))$.
It follows that $\mathcal{Z}(\pi_{d,N,l}^0)$ lifts uniquely to an element of
${\rm CH}(B\times \mathbf{G}(d,N))= {\rm CH}(B)\otimes {\rm CH}(\mathbf{G}(d,N))$.
 We use now  the fact that
  ${\rm CH}(\mathbf{G}(d,N))$ is  generated as an algebra by the Chern classes
  $c_i$ of the rank $d$ quotient bundle $\mathcal{Q}$ to write
   $\mathcal{Z}(\pi_{d,N,l}^0)$
 as a polynomial  in the variables ${\rm pr}_{\mathbf{G}(d,N)}^*c_i$, with coefficients in
 ${\rm CH}(B)$.
 In fact, another  choice of variables is given by
 the relative Chern classes $c_i(\mathcal{X}_{d,N,l}/B),\,i\leq d$ which can be defined as
 the degree $i$ pieces of the Chern polynomial $c(T_{\mathcal{X}_{d,N,l}})c(\pi_{d,N,l}^*T_B)^{-1}$.
 Indeed, the normal bundle exact sequence of $\mathcal{X}_{d,N,l}$ in
 $B\times \mathbf{G}(d,N)$ and Whitney formula provide for $l$ large enough an invertible change of variables between
 the classes  $c_i$  and  $c_i(\mathcal{X}_{d,N,l}/B)$, with coefficients in
 ${\rm CH}(B)$, and of the form $c_i(\mathcal{X}_{d,N,l}/B)=\mu_i c_i+R_i(c_1,\ldots,c_{i-1})$ where
  $\mu_i$ is a  rational number which is nonzero for $l$ large enough, and $R_i$ is a polynomial with coefficients in ${\rm CH}(B)$.
 \begin{rema}\label{remaapriouver}{\rm The classes $c_i(\mathcal{X}_{N,l}/B)$ as defined above restrict  to
 the relative Chern classes $c_i(\pi_{N,l}^0)\in{\rm CH}(\mathcal{X}_{N,l}^0/B)$ by the tangent bundle sequence. The class $c_{d+1}(\mathcal{X}_{N,l}/B)$ can be shown to be equal to the class
 of the critical locus $\Gamma_{crit}=\mathcal{X}_{N,l}\setminus \mathcal{X}_{N,l}^0$. The kernel
 of the restriction map ${\rm CH}(\mathcal{X}_{N,l})\rightarrow {\rm CH}(\mathcal{X}_{N,l}^0)$ is the ideal generated by the class of $\Gamma_{crit}$.}
 \end{rema}
 We thus get a polynomial $P_l$ of $d$ variables $x_i,\,i=1,\ldots,\,d$ with coefficients in ${\rm CH}(B)$, of weighted degree $\leq d$ in the variables $x_i$ of weighted degree $i$, such that
 \begin{eqnarray}\label{eqpouZCI} \mathcal{Z}(\pi_{d,N,l}^0)=P_l(c_1(T_{\mathcal{X}^0_{d,N,l}/B}),\ldots,\,c_d(T_{\mathcal{X}^0_{d,N,l}/B}))
 \,\,{\rm in}\,\,{\rm CH}(\mathcal{X}_{d,N,l}^0).
 \end{eqnarray}

 The main contents of Proposition \ref{proCIr=1} is the fact that $P_l$ has coefficients in $\mathbb{Q}$ (rather than ${\rm CH}(B)$), and that it is independent of $l$. Recall that $B=\mathbf{G}(k_d, V_{d,l})$ so that in  degree $*\leq d$ (which is small compared to $k_d$ and $N$), ${\rm CH}^*(B)$ is freely generated by $$C_i:=c_i(\mathcal{Q}_{k_d}),$$
  where we can of course consider only the $C_j$'s with $j\leq d$.
We can thus write
\begin{eqnarray}\label{eqtardtardtard} P_l=P_l(c_i,C_j),\end{eqnarray} where on the right hand side, $P_l(c_i,C_j)$ is now a polynomial with rational coefficients.

 We now choose   a smooth (nonnecessarily connected) projective algebraic scheme $X_0$ equidimensional of dimension
 $d$ satisfying the conclusion of Proposition
  \ref{lealapacdenori}. Thus, ${\rm CH}^*(X_0)$ contains in degree $*\leq d$ the free polynomial algebra $\mathbb{Q}[c_i,c'_j]_{1\leq i,\,j\leq d}$,
with $c_i=c_i(X_0)$ and $c'_i=f^*c_i(\mathcal{Q})$ for an embedding
$f: X_0\hookrightarrow G(d,N)$.
 (We will allow this embedding to vary later on.)
Let $l$ be large enough so that $f(X_0)$ is defined by degree $l$ equations.
Let $V_{d,l, X_0}\subset V_{d,l}$ be the set of sections of $\mathcal{O}_{G(d,N)}(l)$ vanishing on $X_0$. We now make the same construction and computations  as in
the previous proofs.
Denoting $B_{X_0}\subset B$ the Grassmannian
$\mathbf{G}(k_d,V_{l,d,X_0})$, we  have the inclusion $i:B_{X_0}\hookrightarrow  B$, under which the tautological Chern  classes
$C_j$ of the Grassmannian $B$ restrict to the tautological Chern classes $C_j(B_{X_0})$ of the Grassmannian
$B_{X_0}$,
 and the inclusion
\begin{eqnarray}\label{eqinclusionetalee}( B_{X_0}\times X_0)\setminus \mathcal{C}_{X_0}\subset i^*\mathcal{X}^0_{d,N,l}
\end{eqnarray}
 over
$B_{X_0}$.
We then deduce from (\ref{eqpouZCI}) and (\ref{eqtardtardtard}), using the axioms of Definition \ref{defiaxiom}, that
\begin{eqnarray}\label{eqpullbackaxzero}
P_l(c_i(X_0), C_j)=p_{X_0}^*(\mathcal{Z}({X_0}))\,\,{\rm  in }\,\,{\rm CH}((B_{X_0}\times X_0)\setminus \mathcal{C}_{X_0}),
\end{eqnarray}
 where, as before, $p_{X_0}: (B_{X_0}\times X_0)\setminus \mathcal{C}_{X_0})\rightarrow X_0$ is the natural map. Using the analysis of ${\rm CH}(\mathcal{C}_{X_0})$
already made in the proof of Proposition \ref{lefirstred}, the image
of  ${\rm CH}(\mathcal{C}_{X_0})$ in ${\rm CH}(B_{X_0}\times X_0 )$ is the ideal of
${\rm CH}(B_{X_0}\times X_0)$ generated over ${\rm CH}( B_{X_0}\times X_0)$
by the classes
$C'_j$ for $j=1,\ldots, k_d$, where, using  the notation
$s_i:=s_i(N_{X_0/\mathbf{G}(d,N)}(-l))$, we have as in (\ref{eqclasses})
\begin{eqnarray}\label{eqCprimei}  C'_j=C_j+s_1C_{j-1}+\ldots + s_{j}\in {\rm CH}( B_{X_0}\times X_0).
\end{eqnarray}
Using (\ref{eqCprimei}), we  get that, modulo the ideal generated by the classes $C'_j$, hence equivalently, after restriction to
$(B_{X_0}\times X_0)\setminus \mathcal{C}_{X_0}$, the classes
$C_j$  are given by universal polynomials with
$\mathbb{Q}$-coefficients in the $s_i$.

Recall now  that,  in degree $*\leq d$,  ${\rm CH}^*(X_0\times B_{X_0})$
contains a weighted polynomial ring in the variables $C_k$ coming from $B_{X_0}$, and $c_i(X_0),\,c'_j$ coming from
$X_0$, and of respective weighted degrees $k,\,i,\,j$. The relations
between the classes  $c_i(X_0)$, $c'_i$ and  $s_i$  are given by the normal bundle sequence
of $X_0$ in $\mathbf{G}(d,N)$, which gives
$c(X_0)c(N_{X_0/\mathbf{G}(d,N)})=c(T_{\mathbf{G}(d,N)})_{\mid X_0}$
and thus
\begin{eqnarray}
\label{eqnormbun} s(N_{X_0/\mathbf{G}(d,N)})=c(X_0)s(T_{\mathbf{G}(d,N)})_{\mid X_0},
\end{eqnarray}
where $s$ denotes the total Segre class (which is the inverse of the total Chern class).
It follows that for fixed $l$, the $s_i$'s are given by universal polynomials
$U_i(c_i(X_0),\,c'_j)$ (we also use $c_1(\mathcal{O}_{\mathbf{G}(d,N)}(1)_{\mid X_0})=c'_1$ to deal with the twist of the  normal bundle).

 Combining these observations, we conclude
 that for some universal polynomials $U'_j$ (depending only on $d$, $k_d$, $l$) in the variables
$c_i(X_0),\,c'_j$,
\begin{eqnarray} \label{eqnouveaudecheznew}  p_{B_{X_0}}^*C_i=p_{X_0}^*U'_i(c_p(X_0),c'_q)\,\,{\rm in}\,\,{\rm CH}((B_{X_0}\times X_0)\setminus \mathcal{C}_{X_0})\cong {\rm CH}(B_{X_0}\times X_0)/\langle C'_j\rangle,
\end{eqnarray}
which, combined with
 (\ref{eqpullbackaxzero}), provides
\begin{eqnarray}\label{equnivformula}p_{X_0}^*(P_l(c_i(X_0),U'_j(c_p(X_0),c'_q)))
=p_{X_0}^*(\mathcal{Z}(X_0))\,\,{\rm in}\,\,{\rm CH}((B_{X_0}\times X_0)\setminus \mathcal{C}_{X_0}).
\end{eqnarray}
By proposition \ref{lefirstred}, this relation holds in fact in ${\rm CH}(X_0)$, that is,

\begin{eqnarray}\label{equnivformulaenbas}P_l(c_i(X_0),U'_j(c_p(X_0),c'_q))
=\mathcal{Z}(X_0)\,\,{\rm in}\,\,{\rm CH}( X_0).
\end{eqnarray}.

We argue now using the fact that the right hand side of (\ref{equnivformulaenbas})  depends only on $X_0$ while by changing $f$, the
$c'_q$ can be chosen almost freely in the $\mathbb{Q}$-vector space of weighted degree $q$ homogeneous polynomials in the $c'_j$'s (this is not completely true because
we need to get an embedding $f$ of $X_0$ via a vector bundle with Chern classes $c'_q$ and we want
$f(X_0)$ to be defined by degree $l$ equations, but we easily see that for $l$ large enough, we can choose the
$c'_q$ in the considered subring of  ${\rm CH}(X_0)$ as free variables with respect to polynomials of bounded degree). Thus
the  polynomial  $P_l(c_i(X_0),U'_j(c_p(X_0),c'_q))$, which is a polynomial with
$\mathbb{Q}$-coefficients in the variables $c_i(X_0)$ and  $c'_j$, has to be constant in
the variables $c'_q$.
It remains to see that this actually implies  that $P_l$ is constant in  the variables  $C_j$.
 We use for this the following
\begin{lemm}\label{leallurepol} The polynomials $U'_j$ take the form $U'_j=\mu_j c'_j+Q_j(c'_1,\ldots,c'_{j-1})$  for some universal nonzero constant
$\mu_j$ (depending only on $l,\,d$ and $N$, which are all fixed), where $Q_j$ has coefficients in $\mathbb{Q}[c_i(X_0)]$.
\end{lemm}
\begin{proof} Indeed, from the
relations
$C'_j=\sum_{i=0}^{j}C_{j-i} s_i(N_{X_0/\mathbf{G}(d,N)}(-l))$ for $j\geq 1$, one deduces, by definition of the Segre classes $s_i$, that \begin{eqnarray}\label{eqnewdu11aout} U'_i(c_p(X_0),c'_q)=p_{B_{X_0}}^*C_j=c_j(N_{X_0/\mathbf{G}(d,N)}(-l))\end{eqnarray} modulo the ideal generated by the
$C'_j$'s for $j\geq 1$. Next, by the
exact sequence
$$0\rightarrow T_{X_0}\rightarrow T_{\mathbf{G}(d,N)\mid X_0}\rightarrow N_{X_0/\mathbf{G}(d,N)}
\rightarrow 0,$$
we get
\begin{eqnarray}\label{eqraj20juin19}c(N_{X_0/\mathbf{G}(d,N)})=c(T_{\mathbf{G}(d,N)\mid X_0})s(T_{X_0})
\end{eqnarray}
with $s_0(T_{X_0})=1$ and we observe that the Chern classes $c_j(T_{\mathbf{G}(d,N)})$ satisfy the property

\begin{eqnarray}\label{eqraj20juin191}c_j(T_{\mathbf{G}(d,N)})=\nu_j c_j(\mathcal{Q}) +Q_j(c_1(\mathcal{Q}),\ldots,c_{j-1}(\mathcal{Q}))\end{eqnarray} for some
 universal nonzero constant
$\nu_j$ (depending only on $d,\,N$)  and polynomials $Q_j$.  Restricting (\ref{eqraj20juin191}) to $X_0$,
we get
\begin{eqnarray}\label{eqraj20juin192}c_j(T_{\mathbf{G}(d,N)\mid X_0})=\nu_j c'_j +Q_j(c'_1,\ldots,c'_{j-1})\,\,{\rm in}\,\,{\rm CH}(X_0).\end{eqnarray}
Finally, we have
\begin{eqnarray}\label{eqraj20juin193}c_j(N_{X_0/\mathbf{G}(d,N)}(-l))=c_j(N_{X_0/\mathbf{G}(d,N)}+\sum_{k\geq 1}(-l)^kc_1\mathcal{O}(1)^kc_{j-k}(N_{X_0/\mathbf{G}(d,N)})
\end{eqnarray}
with  $c_1\mathcal{O}(1)=c'_1$. Combining (\ref{eqraj20juin19}), (\ref{eqraj20juin191}), (\ref{eqraj20juin192}) and (\ref{eqraj20juin193}) gives the desired result.
\end{proof}
 We now have the following easy lemma.
 \begin{lemm} \label{letresfacileconst}Let  $F\in \mathbb{Q}[x_1,\ldots,x_d,\,y_1,\ldots,\,y_d]$
 be a polynomial. Let
 $$V_j\in \mathbb{Q}[x_1,\ldots,x_d,\,y_1,\ldots,\,y_d]$$  take the form
 $V_j=\mu_j y_j +Q_j(x_1,\ldots,x_d,\ldots,\,y_1,\ldots,\, y_{j-1})$ for some nonzero $\mu_j\in\mathbb{Q}$ and some polynomials $Q_j$ with $\mathbb{Q}$-coefficients.
 Then if $F$ satisfies $F(x_1,\ldots,\,x_d,\,V_1,\ldots,\,V_d)\in \mathbb{Q}[x_1,\ldots,x_d]\subset \mathbb{Q}[x_1,\ldots,x_d,\,y_1,\ldots,\,y_d]$,
 we have $F\in \mathbb{Q}[x_1,\ldots,x_d]$.
 \end{lemm}
\begin{proof} In fact, a transformation $x_j\mapsto x_j,\,y_i\mapsto U_i$ of the form above can be inverted and even induces an automorphism of $\mathbb{Q}[x_1,\ldots,x_d,\,y_1,\ldots,\,y_d]$ over $\mathbb{Q}[x_1,\ldots,x_d]$.
\end{proof}

We now conclude the proof of Proposition \ref{proCIr=1}. By Lemma \ref{leallurepol}, we can   apply Lemma \ref{letresfacileconst}  to $F=P_l$, $U'_j=V_j$ and we
conclude   that the polynomial $P_l$ has constant coefficients (so the variables $C_i$ from the base do not appear).
It remains to prove that the polynomial $P_l$ is independent of $l$ for large $l$, thus concluding the
proof of Proposition \ref{proCIr=1}.
Choose $l_0$ so that $P_l$ as above is well-defined for $l\geq l_0$. Choose
 $X_0\subset \mathbf{G}(d,N)$ satisfying the conclusion that
 there are no polynomial relations of weighted degree
  $\leq d$ between the Chern classes $c_i(X_0)$. Now, for any $l\geq l_0$,  $X_0$ is defined by degree
$l$ equations and we get
a polynomial with constant coefficients $P_l$   satisfying  by the same proof  as above
 formula (\ref{eqpullbackaxzero}), where
we know
that in  the left hand side, the polynomial $P_l$  has constant coefficients, hence equals
$p_{X_0}^*(P_l(c_i(X_0)))$ in ${\rm CH}(( B_{X_0}\times X_0)\setminus \mathcal{C}_{X_0})$.
By Proposition \ref{lefirstred}, the map $p_{X_0}^*:{\rm CH}(X_0)\rightarrow {\rm CH}((B_{X_0}\times X_0)\setminus \mathcal{C}_{X_0})$ is injective, hence we conclude
that $P_l(c_i(X_0))=\mathcal{Z}(X_0)=P_{l_0}(c_i(X_0))$ for $l\geq l_0$, which implies that
$P_l=P_{l_0}$.
\end{proof}
\subsection{The case of products}
This section is devoted to the generalization of Theorem \ref{theointro1} to products (we refer to the introduction for the motivation of this generalization). We first spell-out the  definition of universally defined cycles for products

\begin{Defi}\label{defiaxiom}  A universally defined cycle $\mathcal{Z}$ on products of $r$ varieties of dimensions $d_1,\ldots,\,d_r$ consists in
the data, for any given $r$ morphisms $$\pi_1:\mathcal{X}_1\rightarrow B_1,\,\ldots,\pi_r:\mathcal{X}_1\rightarrow B_r$$
where all varieties are smooth quasiprojective over $\mathbb{C}$ and the morphisms $\pi_i$ are smooth of relative dimension $d_i$, of a cycle $\mathcal{Z}(\pi_1,\ldots,\pi_r)\in {\rm CH}(\mathcal{X}_1\times\ldots\times \mathcal{X}_r)$ satisfying the following axioms:

(i) {\rm Base change}. For any base change maps $\gamma_i: B'_i\rightarrow B_i$ with $B'_i$ smooth,
with induced morphisms $\gamma'_i:\mathcal{X}'_i\rightarrow \mathcal{X}_i$, $\pi'_i:\mathcal{X}'_i\rightarrow B'_i$, where $\mathcal{X}'_i:=\mathcal{X}_i\times_{B_i}B'_i$, one has $$\mathcal{Z}(\pi'_1,\ldots,\pi'_r)=
(\gamma'_1,\ldots,\gamma'_r)^*\mathcal{Z}(\pi_1,\ldots,\pi_r)\,\,{\rm in}\,\, {\rm CH}(\mathcal{X}'_1\times\ldots\times \mathcal{X}'_r).$$

(ii) {\rm Open inclusion} : For any Zariski open sets $\mathcal{U}_i\subset \mathcal{X}_i$ and restricted morphisms
$\pi_{i,\mathcal{U}_i}:\mathcal{U}_i\rightarrow B_i$, one has
$$\mathcal{Z}(\pi_{1,\mathcal{U}_1},\ldots,\pi_{r,\mathcal{U}_r})=
\mathcal{Z}(\pi_1,\ldots,\pi_r)_{\mid \mathcal{U}_1\times\ldots\times \mathcal{U}_r}\,\,{\rm in}\,\, {\rm CH}(\mathcal{U}_1\times\ldots\times \mathcal{U}_r).$$
\end{Defi}

When the base $B$ is a point, a smooth morphism $\pi:\mathcal{X}\rightarrow B$ is just  a smooth variety $X$ over $\mathbb{C}$ and we will use the notation
$\mathcal{Z}(X_1,\ldots,X_r)$ for $\mathcal{Z}(\pi_1,\ldots,\,\pi_r)$, where $\pi_i$ are the constant morphisms.
When $r=1$, and $d_1=d$, we recover    the notion of universally defined cycle on  varieties of dimension $d$ introduced in Definition \ref{defiuniversintro}.

We will prove in this section Theorem \ref{theosansdiag} and first make a few comments about the statement.

\begin{rema}\label{remasurledeg} {\rm The codimension of the cycle is not mentioned explicitly in the statement
of the theorem, but it is hidden in the  condition on the weighted degrees of the polynomial $P$.
For cycles of codimension $>\sum_i d_i$, our conclusion just says that they are $0$ on products
$X_1\times\ldots\times X_r$ of varieties of dimension $d_i$ defined over a field, which is an empty statement.}
\end{rema}
\begin{rema}{\rm  It is possible that the conclusion of Theorem \ref{theosansdiag}  actually holds true for the full universally defined cycle
$\mathcal{Z}$ on products of families rather than its value on closed points, but we have not been able to prove this. Proposition \ref{proCIr=1} is an indication that a stronger result might hold.
}
\end{rema}
\begin{rema}{\rm  The uniqueness statement is easy, as for smooth projective (nonnecessarily connected) schemes
 $X$ of dimension $d$, there
are no} universal {\rm polynomial  relations of weighted  degree $\leq d$ between the Chern classes
$c_i(X)$  (\cite{milnor}, \cite{thom}).
}
\end{rema}

We now turn to the proof of Theorem \ref{theosansdiag}.
 The proof will use  the following  variant of Theorem \ref{theointro1} proved in the previous section, involving   Definition \ref{defiY} below.
Let $Y$ be a smooth quasiprojective variety over $\mathbb{C}$.
\begin{Defi} \label{defiY} A $Y$-universally defined cycle on varieties of dimension $d$ is the data, for each smooth morphism $\pi: \mathcal{X}\rightarrow B$, of a cycle
$\mathcal{Z}(\pi)\in {\rm CH}(\mathcal{X}\times Y)$, satisfying the
two axioms (i) and (ii) of Definition \ref{defiaxiom} for base changes $B'\rightarrow B$ and open sets $\mathcal{U}\subset \mathcal{X}$.
\end{Defi}
\begin{theo} \label{theoY} For any $Y$-universally defined cycle $\mathcal{Z}$ on varieties of dimension $d$, there exists a uniquely defined
polynomial $P(c_1,\ldots, c_d)$ of weighted degree $\leq d$ in the variables
$c_i$ (of degree $i$), with coefficients in ${\rm CH}(Y)$, such that for any smooth variety
$X$ of dimension $d$ over $\mathbb{C}$,
\begin{eqnarray}\label{eqtheosansdiag} \mathcal{Z}(X)=P(c_1(X),\ldots, c_d(X))\,\,{\rm in}\,\,{\rm CH}(X\times Y).
\end{eqnarray}
\end{theo}
\begin{rema}\label{remadeg} {\rm From the structure of the statement, we note that we can assume that the codimension
of the cycle $\mathcal{Z}$ is $\leq d+{\rm dim}\,Y$, which we will assume below, since in codimension
$t> d+{\rm dim }\,Y$, ${\rm CH}^t(X\times Y)=0$.}
 \end{rema}
Let us first show how the polynomials are determined, as this will be used below. Let
$M$ be the dimension of the space of homogeneous polynomials  of weighted degree
$d$ in the variables $c_1,\ldots,c_d$, with basis  $P_J$ given by degree $d$ monomials. According
to general complex cobordism theory, see \cite{milnor}, \cite{thom}, we can choose $M$ smooth projective varieties $X_i$ of dimension $d$, such that
the Chern numbers $\int_{X_i}P_I(c_1(X_i),\ldots,c_d(X_i))$ give a rank $M$ matrix $M_{iI}$.
If $\mathcal{Z}$ is a $Y$-universally defined cycle of the form
$\mathcal{Z}(X)=\sum_J\alpha_J c_J$, where $\alpha_J\in {\rm CH}(Y)$ and each
$c_J$ is a monomial of weighted degree $\leq d$ in the Chern classes, we
get for any $i$, by pushing forward via the second projection ${\rm pr}_Y:X_i\times Y\rightarrow Y$:
\begin{eqnarray} \label{eqformpoudesc} \sum_{|J|=d}M_{iJ}\alpha_J ={\rm pr}_{Y*}(\mathcal{Z}(X_i)).
\end{eqnarray}
By invertibility of the matrix $M_{iI}$, we conclude that the coefficients $\alpha_J\in{\rm CH}(Y)$ for $|J|=d$ are
universal  linear combinations of the cycles ${\rm pr}_{Y*}(\mathcal{Z}(X_i))$.
Similarly, we get the coefficients  $\alpha_J$  for other weighted degrees ${\rm deg}\, c_J=k<d$ by considering,
for any monomial $P_L$ of degree $d-k$, the $Y$-universally defined cycle
$c_L\cdot \mathcal{Z}$ which to
$\pi:\mathcal{X}\rightarrow B$ associates
$P_L(c_{1}(\mathcal{X}/B),\ldots,c_{d}(\mathcal{X}/B))\cdot \mathcal{Z}\in{\rm CH}(\mathcal{X}\times Y)$.
Then we get as above that if $\mathcal{Z}=\sum_J \alpha_J c_J$, for any smooth projective
 $X$ and any $L$ as above
\begin{eqnarray} \label{eqformpoudesc2}\sum_{|J|=k}M_{XJL}\alpha_J ={\rm pr}_{Y*}(c_L\cdot \mathcal{Z}(X)),
\end{eqnarray}
where $M_{XJL}=\int_{X}c_L(X)c_J(X)$. In order to express $\alpha_J$ using these relations, it thus suffices to exhibit pairs $(X_i,L_i)$ such
that the matrix $M_{iJ}:=M_{X_iJL_i}$ is nondegenerate, which follows again from general complex cobordism theory.

 \begin{proof}[Proof of the implication Theorem \ref{theoY} $\Rightarrow$  Theorem \ref{theosansdiag}]  We do this by induction on $r$. For $r=1$,  Theorem \ref{theosansdiag} is already proved.
Assuming Theorem
\ref{theoY}, and  also Theorem \ref{theosansdiag}   for $r-1$, let $\mathcal{Z}$ be a universally defined cycle for products of $r$ smooth varieties of dimensions $d_1,\ldots,\, d_r$.
Given $\pi_1:\mathcal{X}_1\rightarrow B_1,\,\ldots,\,\pi_{r-1}:\mathcal{X}_{r-1}\rightarrow B_{r-1}$, let
$Y:=\mathcal{X}_1\times\ldots\times \mathcal{X}_{r-1}$. The cycle $\mathcal{Z}$ associates to
any  smooth morphism $\pi_r:\mathcal{X}_r\rightarrow B_r$ of relative dimension $d_r$, with $\mathcal{X}_r$ and $B_r$ smooth, a cycle
$$\mathcal{Z}(\pi_1,\ldots,\pi_r)\in {\rm CH}(\mathcal{X}_1\times\ldots\times\mathcal{X}_r)=
{\rm CH}(Y\times\mathcal{X}_r).$$
This cycle is clearly $Y$-universal in $\mathcal{X}_r$.
Theorem \ref{theoY} says that
there exist well defined cycles $\alpha_I\in {\rm CH}(Y)$ associated to monomials
$I$ of weighted degree $\leq d_r$ such that, for any $X_r$ over $\mathbb{C}$
\begin{eqnarray}
\label{eqrrmoinsun} \mathcal{Z}(\mathcal{X}_1,\ldots,\mathcal{X}_{r-1},\,X_r)=\sum_{I,\,|I|\leq d_r}\alpha_I c_I(X_r)
\end{eqnarray}
in
${\rm CH}(Y\times {X}_r)= {\rm CH}(\mathcal{X}_1\times \ldots\times \mathcal{X}_{r-1}\times {X}_r)$.
We now have
\begin{lemm} \label{leunirmoins1} The cycles $\alpha_I$ above are universally defined cycles on the products
$\mathcal{X}_1\times\ldots\times \mathcal{X}_{r-1}$.
\end{lemm}
\begin{proof} We only have to check axioms (i) and (ii) of Definition
\ref{defiaxiom}. This follows immediately from formulas (\ref{eqformpoudesc}), (\ref{eqformpoudesc2}). Indeed, they express the cycles $\alpha_I$ as linear combinations
of pushforwards $pr^i_{\mathcal{X}_1\times\ldots \mathcal{X}_{r-1}*}(c_L\cdot \mathcal{Z})$,
where we fix a finite number of $X_r^i$ and $pr^i_{\mathcal{X}_1\times\ldots \mathcal{X}_{r-1}}$
is the projection from $\mathcal{X}_1\times\ldots\times  \mathcal{X}_{r-1}\times X_r^i$ to
$\mathcal{X}_1\times\ldots\times  \mathcal{X}_{r-1}$. The fact that $\mathcal{Z}$ satisfies the
base change and open inclusion axioms with respect to $\pi_i$ for $i\leq r-1$ immediately implies that each of these pushforwards satisfies these axioms because pushforward for
proper smooth morphisms commutes with base change of the base (see \cite[Proposition 1.7]{fulton}).
\end{proof}
Lemma \ref{leunirmoins1} and Theorem \ref{theosansdiag}   for $r-1$   thus imply that
the  cycles $\alpha_I(X_1,\ldots,\, X_r)$  are given for products of $r-1$ smooth varieties over $\mathbb{C}$ by polynomials
$Q_I$ with $\mathbb{Q}$-coefficients  in the Chern classes of the $X_j$, $j\leq r-1$.
Combined with (\ref{eqrrmoinsun}), this  concludes the proof.
\end{proof}

The proof of Theorem \ref{theoY} will now follow the same line of reasoning as the proof of Theorem \ref{theointro1}, which was the case where $Y$ is  a point.

Proposition \ref{lefirstred} extends in a straightforward way to this case. Let $E$ be  a rank $k$ vector bundle  over $X$, and $W$ be a vector space  of sections of $E$ generating $E$ at any point.
We have  the universal determinantal hypersurface $\mathcal{C}_X\subset \mathbf{G}(k,W)\times X$.

\begin{lemm}\label{lefirstredY} If the dimension of $W$ is large enough (compared to
the dimension of $X\times Y$), the pull-back and restriction composite map
$${p}_{X\times Y}^*{\rm CH}(X\times Y){\rightarrow} {\rm CH}(((\mathbf{G}(k,W)\times X)\setminus \mathcal{C}_X)\times Y)$$
is injective.
\end{lemm}
\begin{proof}
This follows indeed from Proposition \ref{lefirstred} applied to the vector bundle ${\rm pr}_X^*E$ on
$X\times Y$ and its generating space $W$ of sections.\end{proof}
 We will next  use a slight elaboration of Proposition \ref{proCIr=1} which is needed to overcome the fact that we had assumed previously that the codimension of the cycle
is $\leq d$. The general strategy of the proof is  however very similar.
\begin{proof}[Proof of Theorem \ref{theoY}]
Let $\mathcal{Z}$ be a $Y$-universally defined cycle on varieties of dimension $d$. As before, we
introduce the universal $d$-dimensional smooth complete intersection
$$\pi_{d,N,l}^0:\mathcal{X}_{d,N,l}^0\rightarrow B=\mathbf{G}(k_d,V_{d,N,l})$$
of type $l,\ldots,\,l$ in $G(d,N)$.
As $\mathcal{X}_{d,N,l}^0$ is Zariski open in
$\mathcal{X}_{d,N,l}$ which is fibered into Grassmannians $\mathbf{G}(k_{d}, V_{d,N,l}^{\perp e})$
over $\mathbf{G}(d,N)$, one has surjections:
$${\rm CH}(B\times \mathbf{G}(d,N))\otimes {\rm CH}(Y)\twoheadrightarrow
{\rm CH}(\mathcal{X}_{d,N,l}\times Y) \twoheadrightarrow {\rm CH}(\mathcal{X}_{d,N,l}^0\times Y).$$
This provides us with a polynomial
$P_l(c_j,\,C_i)$  with coefficients in ${\rm CH}(Y)$, in the variables $C_i$ (coming from $B$) and $c_j$ (where as before
the classes $c_j$  will not correspond to the standard classical classes coming from
$\mathbf{G}(d,N)$, but to the Chern classes of the relative tangent bundle of
$\pi_{d,N,l}$) with the property that
\begin{eqnarray}\label{eqpourZCIY}
\mathcal{Z}(\mathcal{X}_{d,N,l}^0)=P_l(c_j(T_{\mathcal{X}_{d,N,l}^0/B}),\,(\pi_{d,N,l}^0)^*(C_i)) \,\,{\rm in}\,\,
{\rm CH}(\mathcal{X}_{d,N,l}^0\times Y).
\end{eqnarray}
Recall the universal polynomials $U'_i(c_p,\,c'_q)$ introduced in the course of
the proof of Proposition \ref{proCIr=1}. They depend only on $d,\,N,\,l$ and have the property that
for any $d$-dimensional smooth quasi-projective algebraic subscheme $X\subset \mathbf{G}(d,N)$ defined over a field,
\begin{eqnarray} c_i(N_{X/\mathbf{G}(d,N)}^*(l))=U'_i(c_p(X),c'_q)\,\,{\rm in}\,\,{\rm CH}(X),
\end{eqnarray}
where $c'_q:=c_q(\mathcal{Q}_{\mid  X})$.
We have the following
\begin{lemm}\label{leproCIY}   For $l$ large enough, the polynomial $P_l$ has the following property:
The polynomial $P'_l$ in the variables $c_i,\,c'_j$ obtained by the formula
$$P'_l=P_l(c_i,\,U'_j(c_p,\,c'_q))$$
belongs to ${\rm CH}(Y)[c_i]$ modulo the ideal generated by the monomials of weighted degree
$\geq d+1$ in the variables  $c_i,\,c'_j$'s (of respective degrees $i,\,j$). In other words, we have
\begin{eqnarray} P'_l=Q_l(c_1,\ldots,\,c_d)\,\,{\rm mod}\,\,\langle c_i,\,c'_j\rangle_{\geq d+1}\label{eqquilnousfaut}
\end{eqnarray}
for some polynomial $Q_l$ of weighted degree $\leq d$, with coefficients in ${\rm CH}(Y)$.
\end{lemm}
\begin{proof} Let $X_0\subset \mathbf{G}(d,N)$ be as in Proposition \ref{lealapacdenori}. For $l$ large enough, $X_0$ is defined schematically by the space
$V_{d,N,l,X_0}\subset V_{d,N,l}$ of degree $l$ equations vanishing on $X_0$, and we have
a surjective differentiation map $V_{d,N,l,X_0}\otimes \mathcal{O}_{X_0}\rightarrow N_{X_0/\mathbf{G}(d,N)}^*(l)$ with universal degeneracy locus
$\mathcal{C}_{X_0}\subset \mathbf{G}(k_d,V_{d,l,X_0})\times X$. Letting $B_{X_0}$ denote $\mathbf{G}(k_d,V_{d,N,l,X_0})$, the restricted projection
$$p_{B_{X_0}}:(B_{X_0}\times X_0)\setminus \mathcal{C}_{X_0}\rightarrow  B_{X_0}$$
 is thus a family
of $d$-dimensional smooth quasiprojective complete intersections of type $l,\ldots,l$ in $\mathbf{G}(d,N)$. We thus
  conclude using the axioms (i) and (ii) in Definition \ref{defiY}  that
\begin{eqnarray}\label{eqderestpourY} \mathcal{Z}(\pi_{d,N,l}^0)_{\mid ((B_{X_0}\times X_0)\setminus \mathcal{C}_{X_0})\times Y}=p_{X_0\times Y}^*\mathcal{Z}(X_0)
\end{eqnarray}
in ${\rm CH}(((B_{X_0}\times X_0)\setminus \mathcal{C}_{X_0})\times Y)$.
Applying (\ref{eqpourZCIY}), the left hand side of (\ref{eqderestpourY}) equals
$P_l(p_{X_0}^*c_i(X_0),p_{B_{X_0}}^*C_j)$, where
$p_{X_0}$ is the projection from $((B_{X_0}\times X_0)\setminus \mathcal{C}_{X_0})\times Y$ to $X_0$ and $p_{B_{X_0}}$ is now the projection from $((B_{X_0}\times X_0)\setminus \mathcal{C}_{X_0})\times Y$ to $B_{X_0}$.
Equation (\ref{eqderestpourY}) holds in  ${\rm CH}(((B_{X_0}\times X_0)\setminus \mathcal{C}_{X_0})\times Y)$, where we know by (\ref{eqnouveaudecheznew}) that
the  relations $$p_{B_{X_0}}^*C_i=p_{X_0}^*c_i(N_{X_0/\mathbf{G}(d,N)}^*(l))=p_{X_0}^*U'_i(c_p(X_0),c'_q)$$ hold. (\ref{eqderestpourY}) thus gives us
\begin{eqnarray}\label{eqpourproCIY2} p_{X_0\times Y}^*(P_l(c_i(X_0),U'_j(c_p(X_0),c'_q)))=p_{X_0\times Y}^*\mathcal{Z}(X_0)
\end{eqnarray}
in ${\rm CH}(((B_{X_0}\times X_0)\setminus \mathcal{C}_{X_0})\times Y)$ and according to Lemma \ref{lefirstredY}, this implies
\begin{eqnarray}\label{eqpourproCIY3} P_l(c_i(X_0),U'_j(c_p(X_0),c'_q)))=:P'_l(c_i(X_0),c'_j)=\mathcal{Z}(X_0)
\end{eqnarray}
in ${\rm CH}(X_0\times Y)$.
In (\ref{eqpourproCIY3}), the right hand side does not depend on the classes  $c'_i\in {\rm CH}(X_0)$, while the left hand side is an element of  the image of
${\rm CH}(Y)[c_i(X_0),\,c'_j]$ in ${\rm CH}(X_0\times Y)$. The smooth quasiprojective scheme  $X_0$ has been chosen in such a way that
the relations in this subring are exactly given by all polynomials of weighted degree $\geq d+1$ in the variables
$c_i,\,c'_j$'s. Hence we conclude that the polynomial $P'_l(c_i(X_0),c'_j)\in {\rm CH}(Y)[c_i(X_0),\,c'_j]$ is constant in the variables  $c'_j$'s modulo the ideal generated by
the monomials of degree $\geq d+1$ in the variables  $c_i,\,c'_j$'s.
\end{proof}
\begin{rema}{\rm The proof above is slightly different from the proof of Proposition \ref{proCIr=1}, but it also proves a slightly weaker statement. With the same arguments as in the previous proof, we could presumably get a more precise statement concerning the polynomials
$P_l$.}
\end{rema}
We now conclude the proof of Theorem \ref{theoY}. Lemma \ref{leproCIY} provides a polynomial
$$Q_l\in{\rm CH}(Y)[c_1,\ldots,\,c_d]$$ of weighted degree $\leq d$ in the $c_i$'s, and we are going to prove that $Q_l=Q $ is independent of $l$ and that for any smooth quasiprojective variety $X$ of pure dimension $d$,
\begin{eqnarray}\label{eqdesired}
\mathcal{Z}(X)=Q(c_i(X))\,\,{\rm in}\,\,{\rm CH}(X\times Y).
\end{eqnarray}
We embed $X$ in $\mathbf{G}(d, N)$ and choose $l$ large enough so that the conclusion
of Lemma \ref{leproCIY} holds. Let as before
$i:B_X=\mathbf{G}(k_d,V_{d,N,l,X})\rightarrow B=\mathbf{G}(k_d,V_{d,N,l})$ be the natural inclusion. Then as we already observed, we  have the open inclusion over
$i(B_{X})$
$$ (B_X\times X)\setminus\mathcal{C}_X\subset i^*\mathcal{X}_{d,N,l}^0$$
of families of smooth complete intersections of type $(l,\ldots,\,l)$ in $\mathbf{G}(d,N)$.
We thus get using the axioms (i) and (ii) of Definition \ref{defiY}
 the equality
 \begin{eqnarray}\label{eqdesired1} p_{X\times Y}^*(\mathcal{Z}(X))=\mathcal{Z}(\pi_{d,N,l}^0)_{\mid  ((B_X\times X)\setminus\mathcal{C}_X)\times Y}\,\,
  {\rm in}\,\, {\rm CH}(((B_X\times X)\setminus\mathcal{C}_X)\times Y).
  \end{eqnarray}
Using (\ref{eqpourZCIY}) and  Lemma \ref{leproCIY},   the right hand side in (\ref{eqdesired1}) is given by
\begin{eqnarray}\label{eqdesired2}\mathcal{Z}(\pi_{d,N,l}^0)_{\mid  ((B_X\times X)\setminus\mathcal{C}_X)\times Y}=p_{X\times Y}^*Q_l(c_1(X),\ldots,\,c_d(X)).
\end{eqnarray}
 Indeed,  in ${\rm CH}(((B_X\times X)\setminus\mathcal{C}_X)$ we have $$p_{B_{X}}^*C_i=p_X^*(U'_i(c_l(X_0),c'_k)),$$  so we get, by restricting (\ref{eqpourZCIY}) to $$((B_X\times X)\setminus\mathcal{C}_X)\times Y\subset \mathcal{X}^0_{d,N,l}\times Y,$$
 \begin{eqnarray}\mathcal{Z}(\pi_{d,N,l}^0)_{\mid  ((B_X\times X)\setminus\mathcal{C}_X)\times Y}= P'_l(p_{X}^*c_i(T_X), p_X^*c'_j))\,\,{\rm in}\,\,{\rm CH}(((B_X\times X)\setminus\mathcal{C}_X)\times Y).
 \end{eqnarray}
  Using Lemma \ref{leproCIY} and the fact that monomials
of degree $>d$ in $c_i,\,c'_j$ vanish on $X\times Y$, we get the desired equality (\ref{eqdesired2}). We thus proved that
$$p_{X\times Y}^*(\mathcal{Z}(X))=p_{X\times Y}^*(Q_l(c_1(X),\ldots,\,c_d(X))\,\,{\rm in }\,\,{\rm CH}(((B_X\times X)\setminus\mathcal{C}_X)\times Y).$$
By Lemma \ref{lefirstredY}, this equality holds already in ${\rm CH}(X\times Y)$, proving
(\ref{eqdesired}) with a polynomial $Q_l$ which could depend on $l$.
Choosing for $X$ a $d$-dimensional scheme $X_0$ such that
there are no polynomial  relation in $H^{2d}(X_0,\mathbb{Q})$ between the $c_i(X_0)$,
the equality $\mathcal{Z}(X_0))=Q_l(c_1(X_0),\ldots,\,c_d(X_0))\,\,{\rm in }\,\,{\rm CH}(X_0\times Y)$
for large $l$ shows that $Q_l$ is independent of $l$, concluding the proof of (\ref{eqdesired}).
\end{proof}

We conclude this section by noting the following variant of Theorem \ref{theosansdiag}. Here we consider the data consisting of a smooth variety $X$ of dimension $d$ and vector bundles $E_1,\ldots,\,E_s$ of respective ranks $r_1,\ldots,\,r_s$  on
$X$. A universally defined cycle for such data associates to each smooth morphism
$\phi:\mathcal{X}\rightarrow B$ of relative dimension $d$ between smooth projective varieties, and vector bundles $\mathcal{E}_1,\ldots,\,\mathcal{E}_s$ on $\mathcal{X}$ of respective ranks $r_1,\ldots,\,r_s$, a cycle
$\mathcal{Z}(\phi,\mathcal{E}_\bullet)\in {\rm CH}(\mathcal{X})$ satisfying the axioms
(i) and (ii) of Definition \ref{defiaxiom}, that is, compatibility with base change and restrictions
to Zariski open subsets.

\begin{theo}\label{theovariant} For any universally defined cycle for  varieties of dimension $d$ and
$s$ vector bundles of ranks $r_1,\ldots,\,r_s$, there is  a uniquely defined
polynomial $P$ of weighted degree $\leq d$ in the variables
$c_i, c_{l_j,j},\,1\leq i\leq d,\, 1\leq j\leq s,\,1\leq l_j \leq {\rm Min}(d,r_j)$, such that
for any $X,\,E_1,\ldots, E_s$ as above,
$$\mathcal{Z}(X,E_\bullet)=P(c_i(X), c_{l_j}(E_j))\,\,{\rm in}\,\,{\rm CH}(X).$$
\end{theo}
\begin{rema} {\rm As in the statement of Theorem \ref{theosansdiag}, we can also improve the result by introducing products
of $k$ varieties each equipped with  vector bundles, and get a similar statement}.
\end{rema}
The proof of Theorem \ref{theovariant} is exactly the same as before, except that we have to work with
 varieties embedded in
products of Grassmannians $\mathbf{G}(d, N)\times \mathbf{G}(r_1,N)\times\ldots\times \mathbf{G}(r_s,N)$. The first embedding
will provide the Chern classes of $X$ and the $i$-th other embedding will provide  the Chern classes of $E_i$.
\section{Universally defined cycles on powers\label{sectionavecpuissance}}
We turn in this section to the study of universally defined cycles on powers of $d$-dimensional smooth varieties, whose precise definition is as follows:
\begin{Defi}\label{defiunipower} A universally defined cycle
$\mathcal{Z}$  on $k$-th powers of varieties of dimension $d$ is
the data, for each smooth  morphism $\mathcal{X}\rightarrow B$ of relative dimension $d$, where $B$ and $\mathcal{X} $ are smooth,
of a cycle
$\mathcal{Z}(\mathcal{X})\in{\rm CH}(\mathcal{X}^{k/B})$ satisfying the following axioms:

(i) If $f: B'\rightarrow B$ is a base change map, with fibered product $\mathcal{X}'=\mathcal{X}\times_{B}B'\rightarrow B'$, ${f'}:\mathcal{X}'\rightarrow \mathcal{X}$,
then $ \mathcal{Z}(\mathcal{X}')=({f'}^k)^*\mathcal{Z}(\mathcal{X})$ in ${\rm CH}(\mathcal{X}^{k/B})$, where
${f'}^k:(\mathcal{X}')^{k/B'}\rightarrow \mathcal{X}^{k/B}$  is  the natural morphism.

(ii) If $\mathcal{X}\rightarrow B$ is as above, and $\mathcal{U}\subset \mathcal{X}$ is a Zariski open set, then
$\mathcal{Z}(\mathcal{U})=\mathcal{Z}(\mathcal{X})_{\mid \mathcal{U}^{k/B}}$ in ${\rm CH}( \mathcal{U}^{k/B})$.
\end{Defi}
As before, the cycle  $\mathcal{Z}(\mathcal{X})$ depends on the morphism $\mathcal{X}\rightarrow B$, even if  the notation does not indicate it.
The principal source of examples of universally defined cycles on powers $X^l$ is the
``natural" cycles  on the Hilbert scheme $X^{[k]}$, assuming it is smooth, that is, in the case of dimension $d=2$ or in any dimension, with $k\leq 3$. Examples of natural  cycles on the Hilbert scheme are as before polynomials in  the Chern classes of the
relative tangent bundle $T_{\mathcal{S}^{[k]/B}}$, but we   also have  the
multiplicity strata and the diagonals or incidence subvarieties in products of Hilbert schemes.
For each partition $I$ of $\{1,\ldots,\,k\}$ (or partition of $k$ given by the multiplicities $i_s=|I_s|$), there
is
an incidence correspondence
$$\Gamma_I\subset X^{l(I)}\times X^{[k]}$$
whose fiber over $(x_1,\ldots,x_{l(I)})$  parameterizes the set of subschemes
of $X$ with associated cycle $\sum_si_sx_s$  (this is the correct definition if the $x_i$'s are  $l(I)$ distinct points of $X$, and it is extended to $X^{l(I)}) $ by taking the Zariski closure). The collection of these correspondences induces the de Cataldo-Migliorini decomposition \cite{deca}.
Being proper over $X^{l(I)} $, the correspondence $\Gamma_I$ induces
a morphism
$$\Gamma_I^*:{\rm CH}(X^{[k]})\rightarrow {\rm CH}(X^{l(I)}).$$

The construction of the Hilbert scheme can be done in the relative setting of  a smooth family of surfaces
$\mathcal{S}\rightarrow B$.
The correspondences above then have  a relative version
$\Gamma_{I/B}$ and we get
universally defined cycles on $l(I)$-th powers of surfaces starting from any ``natural cycle"
on $\mathcal{S}^{[k]/B}$.  One can also introduce products
of correspondences
$\Gamma_{I,I'}=\Gamma_I\times \Gamma_{I'}\subset \mathcal{S}^{l(I)/B}\times_B \mathcal{S}^{l(I')/B}\times  \mathcal{S}^{[k]/B}\times_B\mathcal{S}^{[k]/B}$
and consider
$\Gamma_{I,I'}^*\Delta_{\mathcal{S}^{[k]/B}}$. All these cycles  are easily seen to be universally defined
on powers of surfaces.

Coming back to Definition \ref{defiunipower}, we see that, compared to the previous sections, the novelty is the fact that we have obvious new  universally defined cycles appearing on powers, namely the diagonals. Note that we can also use  the notion  of  universally defined cycle $\mathcal{Z}_I$ on  products of $i_1,\ldots, i_s$-th powers of $d$-dimensional varieties, which is the obvious variant of Definition \ref{defiunipower}.

    In order to deal with the combinatorics of the diagonals, let us introduce some definitions. A partition
$I=\{I_1,\ldots,I_l\}$ of $\{1,\ldots,k\}$ is a decomposition of $\{1,\ldots,k\}$ as a disjoint union
$$\{1,\ldots,k\}=I_1\sqcup\ldots\sqcup I_l,$$
where the $I_i$'s are not empty. We denote the integer $l$ by $l(I)$. To such a partition $I$ is associated a diagonal, respectively a relative diagonal if we work with a morphism $\mathcal{X}\rightarrow B$,
$$\Delta_I\subset  X^k,\,{\rm resp}\,\,\Delta_I\subset \mathcal{X}^{k/B},$$
defined as the set $\{(x_1,\ldots,x_k),\,x_i=x_j\,\,{\rm if}\,\, i,\,j\in I_s\,\,{\rm for}\,\,{\rm some}\,\,s\}$. If
$I=\{\{1\},\ldots,\{k\}\}$, $\Delta_I=X^k$, and if $I=\{\{{1},\ldots,{k}\}\}$, $\Delta_I$ is the small diagonal where all points $x_i$ coincide.
Note that $\Delta_I$ is almost canonically isomorphic to $X^{l(I)}$ (resp. $\mathcal{X}^{l(I)/B})$, where the isomorphism depends in fact of the ordering (indexing) of the set  $\{I_1,\ldots,\, I_{l(I)}\}$. (In practice, we can choose the natural ordering.)  We will thus consider $\Delta_I$ as a morphism from
$X^{l(I)}$ to $X^k$ defined by such an ordering.  This  morphism maps $(x_1,\ldots,x_l)$ to $(x'_1,\ldots,x'_k)$, where
$x'_i=x_s$ whenever $i$ belongs to $I_s$.

Thanks to Theorem \ref{theosansdiag}   which describes universally defined cycles on products, Conjecture  \ref{theavecpowerintro} is equivalent to the following statement.
\begin{conj}\label{theomainwithpowers} Let $\mathcal{Z}$ be a universally defined cycle on $k$-th powers of $d$-dimensional varieties. Then there
exists a unique set of   universally defined cycles $\mathcal{T}_I$ indexed by the partitions $I$ of $\{1,\ldots,k\}$, where each $\mathcal{T}_I$
is universally defined on products of $l(I)$ smooth  varieties of dimension $d$, such that for any variety $X$ over a field,
\begin{eqnarray}\label{feqformulepourcyclepower} \mathcal{Z}({X})=\sum_I\Delta_{I*}(\mathcal{T}_I(X,\ldots,\, X)) \,\,{\rm in}\,\,{\rm CH}(X^k).
\end{eqnarray}
\end{conj}
Here we recall that $\mathcal{T}_I$ provides a cycle in ${\rm CH}(X_1\times \ldots\times X_l)$ for any smooth $d$-dimensional varieties
$X_1,\ldots, \,X_l$, and the meaning of  $\mathcal{T}_I(X,\ldots,\, X)$ is that we apply $\mathcal{T}_I$ to the case where $X_1=\ldots=X_l=X$.

Indeed, Theorem \ref{theosansdiag} applied to each $\mathcal{T}_I$ combined with Conjecture \ref{theomainwithpowers}  provides the desired formula
\begin{eqnarray}\label{feqformulepourcyclepoweravecpoly} \mathcal{Z}(X)=\sum_I\Delta_{I,X*}P_I({\rm pr}_1^*c_{1}(X),\ldots,\,{\rm pr}_1^*c_{d}(X),\ldots,\,{\rm pr}_{l(I)}^*c_{1}(X),\ldots,\,{\rm pr}_{l(I)}^*c_{d}(X) )\,\,{\rm in}\,\,{\rm CH}(X^k),
\end{eqnarray}
where  $P_I$ is a polynomial  in the variables
 $c_{1,1},\ldots,\,c_{d,1},\ldots,\,c_{1,l(I)},\ldots,\,c_{d,l(I)}$, which is of weighted degree $\leq d$ in each set
 of variables $c_{1,s},\ldots,\,c_{d,s}$.

 A cycle of the form (\ref{feqformulepourcyclepoweravecpoly}) will be said standard.

\subsection{An auxiliary  construction \label{sectiontrick}}
We start with a construction which provides for  each universally defined cycle $\mathcal{Z}$ on $k$-th powers of varieties of dimension
$d$  and   each partition $I$ of
$\{1,\ldots,k\}$  a universally defined cycle
$\mathcal{Z}_I$ on products of  $i_1,\ldots,\,i_{l(I)}$-th powers of  $l(I)$ varieties of dimension $d$.
Let $$\mathcal{X}_1\rightarrow B_1,\ldots,\, \mathcal{X}_{l(I)}\rightarrow B_{l(I)}$$ be smooth morphisms of relative dimension $d$.
Consider the family
\begin{eqnarray}
\label{eqfamillytrick} \mathcal{X}_\bullet:=\mathcal{X}_1\times B_2\times\ldots\times B_{l(I)}\sqcup B_1\times \mathcal{X}_2\times\ldots\times B_{l(I)}\sqcup \ldots\sqcup  B_1\times\ldots\times \mathcal{X}_{l(I)}\\
\nonumber \rightarrow B_1\times\ldots\times B_{l(I)}.
\end{eqnarray}
Over a point $(b_1,\ldots,\,b_{l(I)})\in B_1\times\ldots\times B_{l(I)}$,  its fiber is the disjoint union $\mathcal{X}_{b_1}\sqcup\ldots\sqcup \mathcal{X}_{b_{l(I)}}$. In the sequel, if
$X_1,\ldots,\,X_l$ are varieties, we will denote similarly $X_\bullet=\sqcup_iX_i$ (this is the case where the basis  $B_i$  are points).
Let $I=\{I_1,\ldots,\,I_{l(I)}\}$ and $i_s:=|I_s|$. There is a natural inclusion map over $B_{1}\times\ldots\times B_{l(I)}$

\begin{eqnarray}
\label{eqcomponetnttrick}\mathcal{X}_1^{i_1/B_1}\times\ldots\times \mathcal{X}_{l(I)}^{i_{l(I)}/B_{l(I)}}
\hookrightarrow \mathcal{X}_\bullet^{k/(B_1\times\ldots\times B_{l(I)})}
\end{eqnarray}
whose image is an union of connected components of $\mathcal{X}_\bullet$ determined by $I$, and  which we will denote by \begin{eqnarray}\label{eqnameofcomponent}\mathcal{X}_1^{I_1/B_1}\times\ldots\times \mathcal{X}_{l(I)}^{I_{l(I)}/B_{l(I)}}.\end{eqnarray}
  Fiberwise, if $I_1=\{j_{1,1}<\ldots <j_{1,i_1}\}\subset \{1,\ldots,\,k\}$,
..., $I_{l(I)}=\{j_{l(I),1}<\ldots <j_{l(I),i_{l(I)}}\}\subset \{1,\ldots,\,k\}$,  this map sends
\begin{eqnarray}
\label{eqmapI}((x_{1,1},\ldots,x_{1, i_1}),(x_{2,1},\ldots,x_{2, i_2}),\ldots,\,(x_{l(I),1},\ldots,x_{l(I), i_{l(I)}}))\in X_1^{i_1}\times \ldots\times X_{l(I)}^{i_{(l(I)}}\end{eqnarray}
to the  element
$$(x_{j})\in  X_\bullet^k,$$ where, if $j=i_{s,t}\in I_s$, we set
$x_j=x_{s,t}$.

 Given a universally defined cycle $\mathcal{Z}$ on $k$-th powers of varieties of dimension $d$ and a partition $I$ of
$\{1,\ldots,\,k\}$,  we  define, for each families
$\mathcal{X}_1\rightarrow B_1,\ldots, \,\mathcal{X}_{l(I)}\rightarrow B_{l(I)}$ of smooth varieties of dimension $d$, the cycle
\begin{eqnarray}\label{eqZJ} \mathcal{Z}_I(\mathcal{X}_1, \ldots, \mathcal{X}_{l(I)})\in{\rm CH}(\mathcal{X}_1^{i_1/B_1}\times \ldots\times \mathcal{X}_{l(I)}^{i_{l(I)/B_{l(I)}}})
\end{eqnarray}

as the restriction of $\mathcal{Z}(\mathcal{X}_\bullet)$ to the subvariety   $\mathcal{X}_1^{I_1/B_1}\times\ldots\times \mathcal{X}_{l(I)}^{I_{l(I)}/B_{l(I)}}$  of  $\mathcal{X}_\bullet^{k/(B_1\times\ldots\times B_{l(I)})}$ determined  by $I$ that we constructed above.
It is clear that the cycle $\mathcal{Z}_I$ is universally defined on products
$X_1^{i_1}\times\ldots\times X_{l(I)}^{i_{l(I)}}$ of  powers of $l(I)$ smooth varieties of dimension $d$ with the exponents
$i_1,\ldots ,i_{l(I)}$  determined by
$I$.

The next construction goes in the other direction: Given a universally defined cycle $\mathcal{T}$ on  products of $i_1,\ldots, i_{l(I)}$-th powers of $d$-dimensional varieties
with $\sum_si_s=k$,  for any smooth morphism $\mathcal{X}\rightarrow B$ of relative dimension $d$, we get by setting
$\mathcal{X}_1=\mathcal{X},\ldots, \,\mathcal{X}_k=\mathcal{X}$, a cycle
$\mathcal{T}(\mathcal{X}, \ldots, \mathcal{X})\in {\rm CH}(\mathcal{X}^{i_1/B}\times \ldots \times \mathcal{X}^{i_{l(I)}/B})$.
  We will  denote  by
  \begin{eqnarray}
  \label{eqnewpourZIdelta}
  \mathcal{T}^\delta(\mathcal{X})\in {\rm CH}(\mathcal{X}^{k/B})
  \end{eqnarray}
   the restriction of
this cycle to $$\mathcal{X}^{k/B}\subset \mathcal{X}^{i_1/B}\times \ldots \times \mathcal{X}^{i_{l(I)}/B},$$  where
$\mathcal{X}^{k/B}$ is naturally identified with the inverse image of the small diagonal of $B$ in $B^k$ under the morphism
$ \mathcal{X}^{i_1/B}\times \ldots \times \mathcal{X}^{i_{l(I)}/B}\rightarrow B^{l(I)}$. $\mathcal{T}^\delta$ is universally defined on $k$th powers of $d$-dimensional varieties.

Starting from  a universally defined cycle $\mathcal{Z}$ on $k$-th powers of $d$-dimensional varieties, the composition of these two operations produces for each $I$ a
 cycle $\mathcal{Z}_I^\delta$ which  is another  universally defined cycle on $k$-th powers of $d$-dimensional varieties.
Note that for  $I=\{\{1,\ldots,k\}\}$, $\mathcal{Z}_I^\delta=\mathcal{Z}$ using the axioms of Definition \ref{defiunipower}.

Assume that a universally defined cycle $\mathcal{Z}$ on $k$-th powers of $d$-dimensional varieties is standard, hence given by a formula  of the form (\ref{feqformulepourcyclepoweravecpoly}).
Then given a partition $I$ of $\{1,\ldots,\,k\}$,
and varieties ${X}_1,\ldots,{X}_{l(I)}$, we apply formula
(\ref{feqformulepourcyclepoweravecpoly}) to the variety $X_\bullet$
of (\ref{eqfamillytrick}) and restrict to the component (\ref{eqnameofcomponent}).
Then we observe that $$\Delta_J(X_\bullet)\cap  ({X}_1^{I_1}\times\ldots\times {X}_{l(I)}^{I_{l(I)}})=\emptyset$$  if $\Delta_I$ is not  contained in $\Delta_J$,  which is equivalent to the fact that some diagonal identity $x_i=x_j$ satisfied  in $\Delta_J$ is not satisfied in $\Delta_I$, or equivalently, at least one
 $J_s$ is not contained in any $I_t$.

  If, to the contrary,  $\Delta_I$ is contained in $\Delta_J$, (that is, the partition $\Delta_J$ is a refinement of the partition of
$\Delta_I$, meaning that each $J_s$ is contained in a $I_t$), then \begin{eqnarray}
\label{eqdefiDelataIJ}
\Delta_J({X}_\bullet)\cap ({X}_1^{I_1}\times\ldots\times {X}_{l(I)}^{I_{l(I)}})=\Delta_{J,I}\subset {X}_1^{I_1}\times\ldots\times {X}_{l(I)}^{I_{l(I)}}.
\end{eqnarray}
In the last formula, recalling that the partition $J$ refines the partition $I$, it provides a partition of each $I_s$, hence a diagonal for each $X_s^{I_s}$ and taking their product,  this  defines  the desired
diagonal $\Delta_{J,I}\subset X_1^{I_1}\times\ldots\times {X}_{l(I)}^{I_{l(I)}}$. We will view it as   a morphism from
$X_1^{i^J_1}\times\ldots\times X_{l(I)}^{i^J_{l(I)}}$
 to $X_1^{I_1}\times\ldots\times X_{l(I)}^{I_{l(I)}}$. Here $i^J_s$ is the number of elements of the partition $J$ contained in $I_s$.
  It follows from these two observations and from (\ref{feqformulepourcyclepoweravecpoly}) that, for each partition $I$ of $\{1,\ldots,k\}$
\begin{eqnarray}\label{eqnew33}\mathcal{Z}_I(X_1,\ldots,X_{l(I)})=
\\
\nonumber
\sum_{\Delta_I\subset \Delta_J}\Delta_{J,I*}(P_J({\rm pr}_1^*c_1(X_{j(1))},\ldots,\,{\rm pr}_1^*c_d(X_{j(1)}),\ldots, {\rm pr}_{l(J)}^*c_1(X_{j(l(J))}),\ldots,\,{\rm pr}_{l(J)}^*c_d(X_{l(J)}) )
\end{eqnarray}
in ${\rm CH}(X_1^{I_1}\times\ldots\times X_{l(I)}^{I_{l(I)}})$, where the map $j:\{1,\ldots,l(J)\}\rightarrow  \{1,\ldots,l(I)\}$ maps $s$ to $t$ whenever
$J_s\subset I_t$.

We now  restrict to the case where ${X}_i={X}$, and  observe that, when $X_i=X$ for all $i$, $\Delta_{J,I}=\Delta_J$ for those $J$ refining $I$. Then (\ref{eqnew33}) provides for any $I$ the following formula
\begin{eqnarray}
\label{eqpourZavectrick} \mathcal{Z}_I^\delta({X})=\sum_{\Delta_I\subset \Delta_J}\Delta_{J*}(P_J({\rm pr}_1^*c_1(X),\ldots,\,{\rm pr}_1^*c_d(X),\ldots, {\rm pr}_{l(I)}^*c_1(X),\ldots,\,{\rm pr}_{l(I)}^*c_d(X) )
\end{eqnarray}
 in ${\rm CH}({X}^{k})$.

\subsection{Proof of Theorem \ref{corocohchow}\label{subseqnew}}
Let us use the construction developed in the previous section to prove
the following
\begin{theo}\label{corocohchowpasintro}  (Cf.  Theorem  \ref{corocohchow}) Let $\mathcal{Z}$ be  a standard universally defined cycle on
$k$-th  powers
of $d$-dimensional smooth varieties. If $\mathcal{Z}(X)$ is cohomologous to $0$ on $X^k$ for any
smooth $d$-dimensional variety $X$ defined over $\mathbb{C}$, $\mathcal{Z}(X)=0$ in ${\rm CH}(X^k)$ for any
such $X$.
\end{theo}

\begin{proof}
Recall that the universally defined cycles on products $\mathcal{Z}_I$ associated with $\mathcal{Z}$ are
  obtained by computing $\mathcal{Z}$ on $X_1^{i_1}\times\ldots\times X_{l(I)}^{i_{l(I)}}$ seen as a
connected component  of
$X_\bullet^k$ where $X_\bullet=\sqcup X_i$.
It follows from this definition that, if $\mathcal{Z} $   is cohomologous to $0$ on any
smooth $d$-dimensional variety over $\mathbb{C}$, then each $\mathcal{Z}_I$ is cohomologous to $0$
on products of powers $X_1^{i_1}\times\ldots\times X_{l(I)}^{i_{l(I)}}$.
It remains to show that this implies that the polynomials $P_I$ appearing in
(\ref{feqformulepourcyclepoweravecpoly}) are $0$.
We do this by decreasing induction on $l(I)$, using the  formula
(\ref{eqnew33}).
When $l(I)=k$, so $I=\{\{1\},\ldots,\{k\}\}$,
we get in particular
$$\mathcal{Z}_{\{\{1\},\ldots,\{k\}\}}(X_1,\ldots,X_k)=$$
$$P_{\{\{1\},\ldots,\{k\}\}}(
{\rm pr}_1^*c_1(X_1),\ldots,\,{\rm pr}_1^*c_d(X_1),\ldots,\,{\rm pr}_k^*c_1(X_k),\ldots,\,{\rm pr}_k^*c_d(X_k))\,\,{\rm in}\,\,{\rm CH}(X_1\times\ldots\times X_k),$$
and this cycle has to be cohomologous to $0$ on $X_1\times\ldots\times X_k$ for any choice of $X_i$; we already saw (see Proposition \ref{lealapacdenori}) that this implies $P_I=0$.
For the general case, we assume that we already proved that $P_J=0$ for $l(J)>l$. Let $I$ be a partition of $\{1,\ldots,k\}$
with $l(I)=l$.
Formula (\ref{eqnew33}) taken in cohomology then gives, when the varieties  $X_i$  are
projective so that the Gysin morphism $\Delta_{I*}$ is well defined on cohomology
$$[\mathcal{Z}_I(X_1,\ldots,\,X_{l(I)})]=\Delta_{I*}P_I({\rm pr}_1^*c_1(X_1),\ldots,\,{\rm pr}_1^*c_d(X_1),\ldots,\,{\rm pr}_{l(I)}^*c_1(X_{l(I)}),\ldots,\,{\rm pr}_l^*c_d(X_{l(I)}))$$
in $H^*(X_1\times\ldots\times X_k,\mathbb{Q})$.
Here the map $\Delta_I$  is the diagonal inclusion of $X_1\times\ldots\times X_{l(I)}$ in $X_1^{I_1}\times\ldots \times X_l^{I_{l(I)}}$ determined by $I$.
The vanishing of $[\mathcal{Z}_I(X_1,\ldots,\,X_{l(I)})]$ and the injectivity of the Gysin morphism
$\Delta_{I*}$ for projective varieties  $X_i$  then implies that $$P_I({\rm pr}_1^*c_1(X_1),\ldots,\,{\rm pr}_1^*c_d(X_1),\ldots,\,{\rm pr}_{l(I)}^*c_1(X_{l(I)}),\ldots,\,{\rm pr}_{l(I)}^*c_d(X_{l(I)}))$$ is cohomologous to  $0$ on
$X_1\times \ldots \times X_{l(I)}$ for any smooth $d$-dimensional  projective varieties $X_1,\ldots, X_{l(I)}$, so that $P_I=0$, by Proposition \ref{lealapacdenori} again.
 \end{proof}
 \subsection{Application to uniqueness \label{subsecunique}}
Using the formalism  above, we prove the following uniqueness result.
\begin{theo}\label{propunique} Let $d$ and $k$ be fixed and let $Y$ be a smooth variety over $\mathbb{C}$. Consider
a collection $P_\bullet=(P_I)$ of polynomials with coefficients in ${\rm CH}(Y)$, where $I$ runs through the set of
partitions of $\{1,\ldots,k\}$, and each $P_I$ is a polynomial in the  variables
$$x_{1,1},\ldots, x_{d,1},\ldots,\,x_{1,l(I)},\ldots,\,x_{d,l(I)},$$  of weighted degree $\leq d$ in each set of variables
$x_{1,j},\ldots,\,x_{d,j}$ (where  ${\rm deg}\,x_{i,j}=i$).
Then the $Y$-universal cycle  $\mathcal{Z}_{P_\bullet}$ on $k$-th powers of $d$-dimensional varieties defined by
\begin{eqnarray}\label{eqZPstand}
\mathcal{Z}(X)=\sum_I\Delta_{I*}P_I({\rm pr}_1^*c_{1}(X),\ldots, {\rm pr}_1^*c_{d}(X),\ldots,\,{\rm pr}_{l(I)}^*c_{1}(X),\ldots, {\rm pr}_{l(I)}^*c_{d}(X))\in {\rm CH}(Y\times X^k)
\end{eqnarray}
 vanishes for any smooth $d$-dimensional variety $X$ over $\mathbb{C}$ if and only if $P_\bullet=0$.
\end{theo}
\begin{rema}{\rm Even in the case where $Y$ is a point, Theorem \ref{propunique} may seem a bit surprising, especially when  combined with Theorem \ref{corocohchowpasintro}. Indeed, consider for example the case of standard universal $0$-cycles on $X^k$. If $X$ is smooth projective and connected, then the cohomology class of a $0$-cycle on $X$ is encoded by its degree. It follows that
all $0$-cycles indexed by partitions and Chern polynomials appearing  in formula (\ref{feqformulepourcyclepoweravecpoly}) have proportional cohomology classes in $X^k$. So for given connected $X$, the statement is wrong. There are two reasons  behind Theorem \ref{propunique}: first of all $X$ is not connected and this is crucial since then the group $H_0(X^k,\mathbb{Z})$ of cycle classes of $0$-cycles encodes the connected components of $X^k$ and in particular the diagonals. The second point   is general complex cobordism, which says (for $k=1$) that there are no {\it universal} polynomial relations between Chern numbers, that is  relations satisfied by all $X$'s.  }
\end{rema}
\begin{proof}[Proof of Theorem \ref{propunique}] Assume $P_\bullet$ satisfies the property that the associated cycle
$\mathcal{Z}_{P_\bullet}(X^k)$ given by formula (\ref{eqZPstand}) vanishes for smooth $d$-dimensional varieties
$X$ over $\mathbb{C}$. As in the previous proof, this implies as well by definition that each
$\mathcal{Z}_{P_\bullet,I}(X_1,\ldots,\, X_{l(I)})$ vanishes for any smooth $d$-dimensional varieties $X_1,\ldots,\,X_{l(I)}$.
 We use now formula (\ref{eqnew33}). Next, as in the previous proof, this  implies by decreasing induction on
 $l=l(I)$  that, denoting by
 $\Delta_{I_\bullet}$ the inclusion of
 $X_1\times\ldots\times X_{l(I)}$ in $X_1^{I_1}\times\ldots X_{l(I)}^{I_{l(I)}}$ given by the small diagonal inclusion on each summand, each cycle
  $$\Delta_{I_\bullet,*}P_I({\rm pr}_1^*c_{1}(X_1),\ldots,\,{\rm pr}_1^*c_{d}(X_1),\,\ldots,\,{\rm pr}_{l(I)}^*c_{1}(X_{l(I)}),\ldots,\,{\rm pr}_{l(I)}^*c_{d}(X_{l(I)}))$$ vanishes in ${\rm CH}(Y\times X_1^{I_1}\times\ldots\times X_{l(I)}^{I_{l(I)}})$ for each set of smooth $d$-dimensional varieties $X_1,\ldots,\,X_{l(I)}$.
  Arguing inductively on $l(I)$ and replacing $Y$  by $Y':=Y\times X_1^{i_1}\times\ldots\times X_{l(I)-1}^{i_{l(I)-1}}$, we conclude that it suffices to consider the case
  where $l(I)=1$, and  $P$ is a polynomial in $c_1,\ldots,\,c_d$ of weighted degree $\leq d$ with coefficients in
  ${\rm CH}(Y')$. We then have to prove that if $P\not=0$, then for any $k>0$, there exists a smooth
   $d$-dimensional variety $X$ such that
  $\Delta_*P(c_1(X),\dots,c_d(X))\not=0$ in ${\rm CH}(Y'\times X^k)$.
  As $\Delta_*$ is injective when $X$ is projective, we are finally reduced to  the statement for $k=1$, which is
  proved in  \cite{thom}, \cite{milnor}.
\end{proof}

CNRS, Institut de math\'{e}matiques de Jussieu-Paris rive gauche

 claire.voisin@imj-prg.fr
    \end{document}